\documentclass[12pt]{amsart}
\usepackage{url}
\usepackage{graphicx,amsmath,amssymb,url,pgf, tikz}
\usepackage[all]{xy}
\usepackage{enumitem}
\usepackage[margin=3.5cm]{geometry}

\usepackage{amsthm}
\usepackage{todonotes}

\numberwithin{equation}{section}
\newtheorem{prop}{Proposition}[section]
\newtheorem{thm}[prop]{Theorem}
\newtheorem{cor}[prop]{Corollary}
\newtheorem{lem}[prop]{Lemma}

\newtheorem{open}{Open Problem}

 \theoremstyle{definition}
 \newtheorem{defn}[prop]{Definition} 

 \theoremstyle{definition}
 \newtheorem{remark}[prop]{Remark} 

 \theoremstyle{definition}
 \newtheorem{algorithm}[prop]{Algorithm} 

 \theoremstyle{definition}
 \newtheorem{example}[prop]{Example} 

 \newcommand{\lessthan}{\prec}
 \newcommand{\MacdonaldMap}{M}
 \newcommand{\TransitionMap}{T_{\pi}}
 \newcommand{\BoundedPairTransMap}{BT_{\pi}}
 \newcommand{\RT}{RT_{\pi}}
 
\newcommand{\Bpoly}{\mathfrak{F}} 
\newcommand{\SpecializedBpoly}{f} 

\newcommand{\codes}{\mathcal{C}}
\newcommand{\comaj}{\operatorname{comaj}}
\newcommand{\given}{\ : \ }

\newcommand{\pic}{\begin{tikzpicture}}
\newcommand{\epic}{\end{tikzpicture}}

\newcommand{\Z}{\mathbb{Z}}

\DeclareMathOperator{\id}{id}
\DeclareMathOperator{\Inv}{Inv}
\DeclareMathOperator{\BoundedPairs}{BoundedPairs}
\DeclareMathOperator{\rpp}{rpp}

\newcommand{\Tset}{\mathcal{U}}
\newcommand{\Xset}{\mathcal{X}}
\newcommand{\Yset}{\mathcal{Y}}

\newcommand{\bfa}{\mathbf{a}}
\newcommand{\bfb}{\mathbf{b}}

\newcommand{\deleted}{\mathsf{deleted}}
\newcommand{\bumped}{\mathsf{bumped}}
\newcommand{\outcome}{\mathsf{outcome}}

\newcommand{\dvalue}{k}
\newcommand{\rp}{\mathcal{RP}}
\newcommand{\cd}{cD}

\newcommand{\defect}{\mathrm{Defect}}

\renewcommand{\d}{\epsilon}

\usepackage{mathrsfs} 

\newcommand{\Push}{\mathscr{P}}
\newcommand{\Bump}{\mathscr{B}}
\newcommand{\Delete}{\mathscr{D}}
\newcommand{\Insert}{\mathscr{I}}

\newcommand{\q}{\mathfrak{q}}

\newlength{\cellsize} 
\newsavebox{\cell}
\sbox{\cell}{\begin{picture}(18,18)
\put(0,0){\line(1,0){18}}
\put(0,0){\line(0,1){18}}
\put(18,0){\line(0,1){18}}
\put(0,18){\line(1,0){18}}
\end{picture}}
\newcommand\cellify[1]{\def\thearg{#1}\def\nothing{}%
\ifx\thearg\nothing
\vrule width0pt height\cellsize depth0pt\else
\hbox to 0pt{\usebox{\cell} \hss}\fi%
\vbox to \cellsize{
\vss
\hbox to \cellsize{\hss$#1$\hss}
\vss}}
\newcommand\tableau[1]{\vtop{\let\\\cr
\baselineskip -16000pt \lineskiplimit 16000pt \lineskip 0pt
\ialign{&\cellify{##}\cr#1\crcr}}}
\usepackage{bbm}

\DeclareSymbolFont{bbsymbol}{U}{bbold}{m}{n}
\DeclareMathSymbol{\Fact}{\mathbin}{bbsymbol}{"21}

\newcommand{\elbows}{\mathbin{\text{\rotatebox[origin=c]{45}{$\asymp$}}}}

\begin{document}


\title{A bijective proof of Macdonald's reduced word formula}
\date{\today}

\author[Billey]{Sara C.~Billey}
\address{Sara Billey, Department of Mathematics,  University of Washington, Box 354350, Seattle, WA 98195, USA}
\thanks
{Billey was partially supported by grant DMS-1101017 from
the NSF} \email{billey@math.washington.edu}

\author[Holroyd]{Alexander E.~Holroyd}
\address{Alexander E.~Holroyd, Microsoft Research, 1 Microsoft Way, Redmond, WA 98052, USA}
\email{holroyd@microsoft.com}

\author[Young]{Benjamin Young}
\address{Benjamin Young, Department of Mathematics, 1222 University of Oregon, Eugene, OR 97403, USA}
\email{bjy@uoregon.edu}

\begin{abstract}
  We give a bijective proof of Macdonald's reduced word
  identity using pipe dreams and Little's bumping
  algorithm. This proof extends to a principal
  specialization due to Fomin and Stanley.  Such a proof
  has been sought for over 20 years.  Our bijective tools
  also allow us to solve a problem posed by Fomin and
  Kirillov from 1997 using work of Wachs, Lenart, Serrano
  and Stump.  These results extend earlier work by the
  third author on a Markov process for reduced words of
  the longest permutation.
\end{abstract}

\maketitle

\section{Introduction}\label{s:intro}

Macdonald gave a remarkable formula connecting a weighted
sum of reduced words for a permutation $\pi$ with the
number of terms in a Schubert polynomial
$\mathfrak{S}_\pi(x_1,\ldots, x_n)$.  For a permutation
$\pi\in S_n$, let $\ell(\pi)$ be its inversion number and
let $R(\pi)$ denote the set of its reduced words. (See
Section 2 for definitions.)

\begin{thm}[Macdonald {\cite[(6.11)]{M2}}]  \label{t:macdonald}
  Given a permutation $\pi \in S_n$ with $\ell(\pi)=p$, one has
\begin{equation}\label{e:macdonald.formula}
\sum_{(a_1,a_2,\ldots, a_p) \in R(\pi)} a_1\cdot a_2 \cdots a_p \ = p! \
\mathfrak{S}_\pi(1,\ldots, 1).
\end{equation}
\end{thm}

For example, the permutation $[3,2,1] \in S_3$ has 2
reduced words, $R([3,2,1])=\{(1,2,1), (2,1,2)\}$.  The
inversion number is $\ell([3,2,1])=3$, and the Schubert
polynomial $\mathfrak{S}_{\pi}(x_1,x_2,x_3)$ is the single
term $x_1^2x_2$. We observe that Macdonald's formula holds:
$1\cdot 2\cdot 1 + 2 \cdot 1 \cdot 2 =3!\cdot 1$.

In this paper, we give a bijective proof of
Theorem~\ref{t:macdonald}. Such a proof has been sought for
over 20 years.  It has been listed as an open problem in
both \cite{Fomin-Kirillov} and \cite{Stanley.perms}. Fomin
and Sagan have stated that they have a bijective proof, but
that it is unpublished due to its complicated nature; see
\cite{Fomin-Kirillov}.  Moreover, we give several
generalizations as discussed below.  Our proof builds on
the work of the third author on a Markov process on reduced
words for the longest permutation \cite{BY2014}.

The Schubert polynomial $\mathfrak{S}_\pi$ can be expressed
as a sum over \emph{reduced pipe dreams} (or \emph{RC
graphs}) corresponding to $\pi$, and its evaluation at
$(1,\ldots,1)$ is simply the number of such pipe dreams.
(See Section 2 for definitions, and
\cite{LS1,BJS,Fomin-Stanley,FK,billey-bergeron} for history
and proofs.) Thus, the right side of
\eqref{e:macdonald.formula} is the number of pairs
$(\mathbf{c},D)$, where $\mathbf{c}=(c_1,\ldots, c_p)$ is a
word with $1 \leq c_i \leq i$ for each $i$, and $D$ is a
pipe dream for $\pi$.  A word $c$ with this property is
sometimes called a \emph{sub-staircase word}.  The left
side is the number of pairs $(\mathbf{a},\mathbf{b})$ where
$\mathbf{a}\in R(\pi)$ and $\mathbf{b}$ is a word
satisfying $1\leq b_i\leq a_i$ for each $i=1,\ldots,p$. Our
bijection is between pairs $(\mathbf{a},\mathbf{b})$ and
$(\mathbf{c},D)$ that satisfy these conditions.  The
bijection and its inverse are presented in the form of
explicit algorithms.  Moreover, both maps are uniform over
the permutation $\pi$ in the sense that they have natural
descriptions that explicitly involve only
$(\mathbf{a},\mathbf{b})$ (respectively, $(\mathbf{c},D)$),
and not $\pi$ (although of course $\pi$ can be recovered
from $\mathbf{a}$ or
$D$).
Indeed, if we
interpret permutations $\pi\in S_n$ as permutations of $\Z$
that fix all but finitely many elements, then our maps do
not even explicitly involve $n$.

The outline of the bijection is quite simple given some well-known
properties of Schubert polynomials, together with the \emph{bumping
algorithm} for reduced words.  The bumping algorithm is an important
tool for studying reduced words, originally introduced and developed by
Little \cite{little2003combinatorial} and further studied by Garsia in
\cite{saga}. These properties and objects will be defined in
Section~\ref{s:background}.

In the first step, we give a modification of Little's
bumping algorithm that also acts on pipe dreams, and use it
to give a bijective interpretation to the
Lascoux-Sch\"utzenberger transition equation for Schubert
polynomials.  Essentially the same construction has been
given by Buch \cite[p.11]{Knutson.2012}. The key idea is to
iteratively apply the corresponding transition map to $D$
until we reach the empty pipe dream, while recording a
sequence of instructions that encode which
inversions/insertions are needed in order to reverse the
process. We call the resulting sequence a transition chain,
denoted $Y(D)$.

\begin{figure}
\begin{center}
  \includegraphics[width=\textwidth]{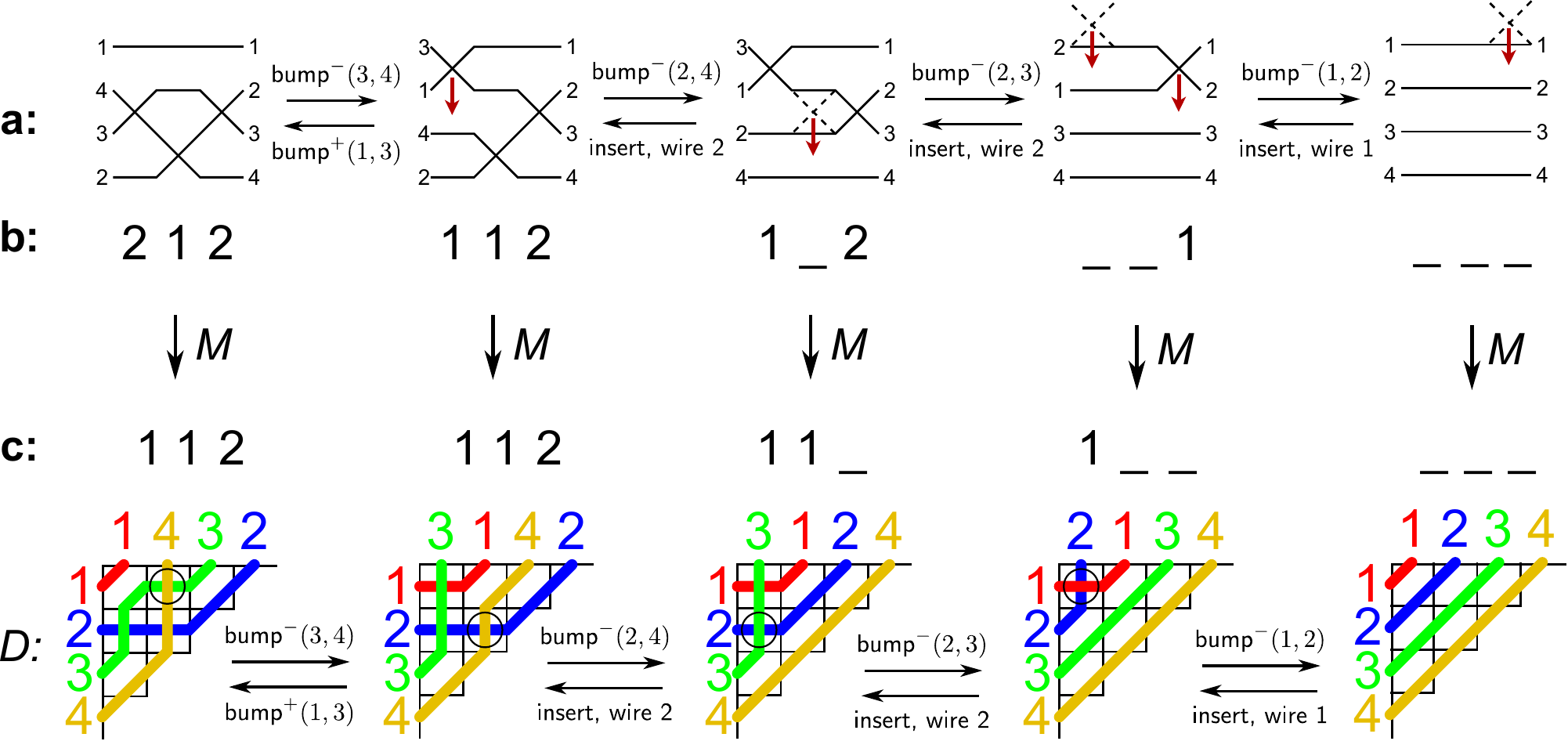}
\caption{An example of the bijection $\MacdonaldMap$ for $\pi=[1,4,3,2]$
  where the pair $(\mathbf{a}, \mathbf{b})$ is mapped to $(\mathbf{c},
  D)$ with $\mathbf{a} =(2,3,2)$, $\mathbf{b}=(2, 1, 2)$,
  $\mathbf{c}=(1,1,2)$, and $D$ is the pipe dream in the top left
  corner of the picture.  Its transition chain is $Y(D)=((1,3), (2,2),(2,2),(1,1))$.
  Each vertical pair in the picture is also demonstrating the
  bijection for a different permutation; note the permutations on the
  wires agree on the vertical pairs.}
\label{fig:bijection}
\end{center}
\end{figure}

Next we apply
the
bumping algorithm on
reduced words and their wiring diagrams, using the reverse
of the transition chain $Y(D)$ to provide instructions. The
word $\mathbf{c}$ tells us where to insert new crossings;
when adding the $i$th new crossing it should become the
$c_i$th entry in the word. The height of the added crossing
is determined by placing its feet on the wire specified by
the corresponding step of the transition chain. Each new
crossing is immediately pushed higher in value, initiating
a Little bump. The result is a reduced wiring diagram for
$\pi$ corresponding to a reduced word
$\mathbf{a}=(a_1,a_2,\ldots, a_p)$.  If we keep track of
how many times each column is pushed in the bumping
processes, we obtain a word $\mathbf{b}=(b_1,\ldots,b_p)$
of the same length such that $a_i \geq b_i$ for all $i$, as
required.  See Figure~\ref{fig:bijection} for an
illustration of the algorithm.  It turns out that each step
is reversible.

Our bijective proof extends to a $\q$-analog of
\eqref{e:macdonald.formula} that was conjectured by Macdonald and
subsequently proved by Fomin and Stanley. To state this formula, let
$\q$ be a formal variable. Define the $\q$-analog of a positive integer
$k$ to be $[k]=[k]_{\q}:=1+\q+\q^2+\cdots + \q^{k-1}$. The $\q$-analog of the
factorial $k!$ is defined to be $[k]\Fact=[k]_\q\Fact := [k][k-1]\cdots
[1]$. (We use the blackboard bold symbol $\Fact$ to distinguish it
from the ordinary factorial, and the symbol $\q$ for the formal variable
to avoid notation conflicts.) For $\mathbf{a}=(a_1,a_2,\ldots, a_p)
\in R(\pi)$, define the co-major index to be the sum of the ascent
locations:
$$\comaj(\mathbf{a}) := \sum_{\substack{1 \leq i < p: \\ a_i < a_{i+1}}} i.$$

\begin{samepage}
\begin{thm}[Fomin and Stanley {\cite{Fomin-Stanley}}] \label{t:macdonald.q.analog}
  Given a permutation $\pi \in S_n$ with $\ell(\pi)=p$, one has
  \begin{equation}\label{e:fomin.stanley.formula}
  \sum_{\mathbf{a}=(a_1,a_2,\ldots, a_p) \in R(\pi)}  [a_1]\cdot [a_2] \cdots [a_p]\ \q^{\comaj(\mathbf{a})} \
  = {[p]\Fact} \,   \mathfrak{S}_\pi(1,\q,\q^2,\ldots, \q^{n-1}).
\end{equation}
\end{thm}
\end{samepage}

Continuing with the example $\pi=[3,2,1]$, we observe that
the $\q$-analog formula indeed holds: $[1]\cdot [2]\cdot
[1]\q + [2] \cdot [1] \cdot [2]\q^2 = (1+\q)\q + (1+ \q)^2\q^2 =
(1+\q+\q^2)(1+\q)\q = [3]{\Fact} \cdot \mathfrak{S}_{[3,2,1]}
(1,\q,\q^2) $.

In 1997, Fomin and Kirillov published a further extension
to Theorem~\ref{t:macdonald.q.analog}.  They interpreted
the right side of the formula in terms of reverse plane
partitions, and asked for a bijective proof.  See
Theorem~\ref{t:fomin.kirillov.2}.  Using our methods
together with results of Lenart \cite{Lenart.2004}, and
Serrano and Stump
\cite{Serrano.Stump.FPSAC,Serrano.Stump.2012}, we
provide a bijective proof.

We want to comment briefly on how the bijections in this paper were
found. We were fully aware of Little's bumping algorithm so we hoped
it would play a role.  Many details of the exact formulation we
describe here were found through extensive experimentation by hand and
by computer. Experimentally, we found the transition chains to be the
key link between a bounded pair and its image under $M$.  As the proof
was written up, we chose to suppress the  dependence on the
transition chains in favor of clearer descriptions of the maps.

The outline of the paper is as follows.  In
Section~\ref{s:background}, we give the notation and background
information on reduced words, Schubert polynomials, Little bumps,
etc. The key tool for our bijection comes from the Transition Equation
for Schubert polynomials.  We give a bijective proof of this equation
in Section~\ref{s:trans}.  In Section~\ref{s:trans.bounded.pairs}, we
extend the Transition Equation to bounded pairs, by which we mean
pairs $(\bfa,\bfb)$ satisfying $1\leq b_{i}\leq a_{i}$, as discussed
above. In Section~\ref{s:bij.proof}, we spell out the main bijection
proving Theorem~\ref{t:macdonald}. The principal specialization of
Macdonald's formula given in Theorem~\ref{t:macdonald.q.analog} is
described in Section~\ref{s:specialization}, along with some
interesting properties of the co-major index on reduced words.  In
Section~\ref{s:fk}, we discuss the Fomin-Kirillov theorems and how our
bijection is related to them.  Finally, in Section~\ref{s:future} we
discuss some intriguing open problems and other formulas related to
Macdonald's formula.

\section{Background}\label{s:background}

\subsection{Permutations}\label{ss:permutations}
We recall some basic notation and definitions relating to
permutations which are standard in the theory of Schubert
polynomials.  We refer the reader
to~\cite{LS1,M2,manivel-book} for more information.

Let $S_n$ be the symmetric group of all permutations
$\pi=[\pi(1),\ldots,\pi(n)]$ of $\{1,\ldots,n\}$.  An \emph{inversion}
of $\pi\in S_n$ is an ordered pair $(i,j)$, such that $ i < j $ and
$\pi(i) > \pi(j)$. The \emph{length} $\ell(\pi)$ is the number of
inversions of $\pi$.  We write $t_{ij}$ for the transposition which
swaps $i$ and $j$, and we write $s_i = t_{i, i+1}$ $(1 \leq i \leq
n-1)$.  The $s_i$ are called \emph{simple transpositions}; they
generate $S_n$ as a Coxeter group.  Composition of permutations is
defined via $\pi\tau(i):=\pi(\tau(i))$.

An alternate notation for a permutation $\pi \in S_n$ is its
\emph{Lehmer code}, or simply its \emph{code}, which is the $n$-tuple
\[
    (L(\pi)_1, L(\pi)_2, \ldots, L(\pi)_n)
\]
where $L(\pi)_i$ denotes the number of inversions $(i,j)$ with first coordinate $i$.  Note, $0 \leq L(\pi)_i \leq n-i$ for
all $1 \leq i \leq n$.  The permutation $\pi$ is said to be
\emph{dominant} if its code is a weakly decreasing sequence.

\subsection{Reduced words}\label{ss:reduced words}

A \emph{word} is a $k$-tuple of integers.  The \emph{ascent set} of a
word $\mathbf{a}=(a_1, \ldots, a_{k})$ is $\{i: a_{i}<a_{i+1} \}
\subseteq \{1,\ldots, k-1 \}$.  The \emph{descent set} of $\mathbf{a}$
is the complement.

Let $\pi \in S_n$ be a permutation.  A \emph{word for $\pi$} is a word
$\mathbf{a}= (a_1, \ldots, a_{k})$ such that $1 \leq a_i < n$ and
\[
    s_{a_1} s_{a_2} \ldots s_{a_{k}} = \pi.
\]
If $k = \ell(\pi)$, then we say that $\mathbf{a}$ is a
\emph{reduced word} for $\pi$.  The reduced words are
precisely the minimum-length ways of representing $\pi$ in
terms of the simple transpositions.  For instance, the
permutation $[3,2,1] \in S_3$ has two reduced words:
(1,2,1) and (2,1,2).  The empty word $()$ is the unique
reduced word for the identity permutation $[1,2,\ldots, n]
\in S_{n}$.

Write $R(\pi)$ for the set of all reduced words of the
permutation $\pi$.  The set $R(\pi)$ has been extensively
studied,
in part
due to interest in Bott-Samelson varieties
and Schubert calculus. Its size has an interpretation in
terms of counting standard tableaux and the Stanley
symmetric functions~\cite{LS1,little2003combinatorial,S3}.

Define the \emph{wiring diagram} for a word
$\mathbf{a}=(a_1,\ldots,a_k)$ as follows. First, for $0
\leq t \leq k$, define the permutation $\pi_t \in S_n$
\emph{at time} $t$ by
\[
    \pi_t =  s_{a_1}s_{a_1}\cdots s_{a_t}.
\]
So $\pi_0$ is the identity, while $\pi_k=\pi$.  The
$i$-\emph{wire} of $\mathbf{a}$ is defined to be the
piecewise linear path joining the points
$(\pi^{-1}_t(i),t)$ for $0 \leq t \leq k$.  We will
consistently use ``matrix coordinates'' to describe wiring
diagrams, so that $(i,j)$ refers to row $i$ (numbered from
the top of the diagram) and column $j$ (numbered from the
left).  The \emph{wiring diagram} is the union of these $n$
wires.  See Figure~\ref{fig:wiring.diag} for an example.

\begin{figure}
\centering
\raisebox{2mm}{\includegraphics[width=1.7in]{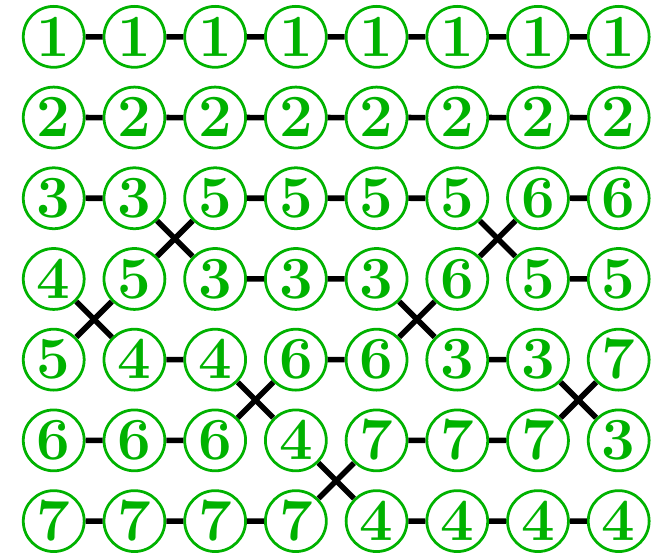}}
\hspace{.2in}
\includegraphics[width=1.7in]{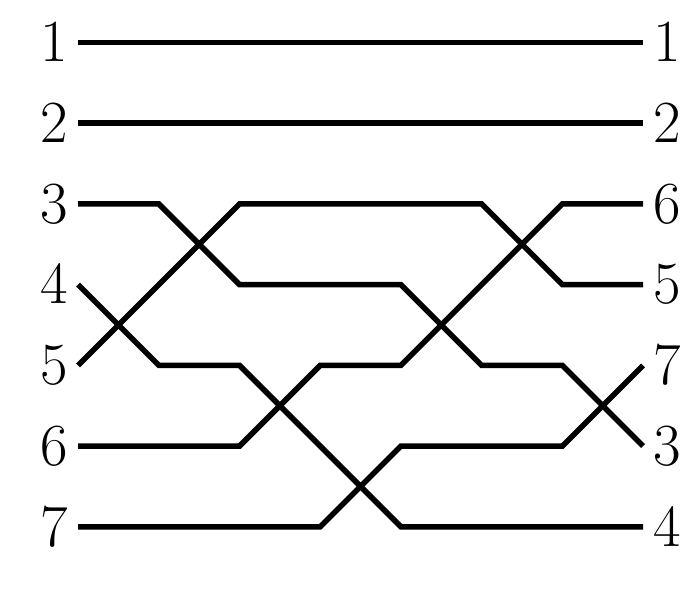}
\hspace{.2in}
\includegraphics[width=1.7in]{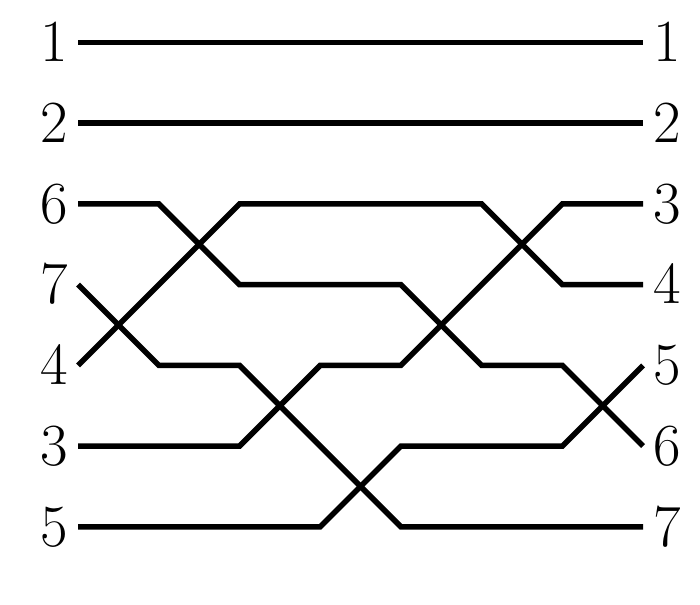}

\caption{The wiring diagram for the reduced word $(4, 3,5, 6, 4, 3,5)
  \in R([1,2,6,5,7,3,4])$ annotated in three different ways: with the intermediate permutations $\pi_t$,
   the left-labeling, and the right-labeling.
   The crossings in columns 2 and 6 are both at row 3.}
\label{fig:wiring.diag}
\end{figure}

For all $t \geq 1$, observe that between columns $t-1$ and
$t$ in the wiring diagram for $\bfa$, precisely two wires
$i$ and $j$ intersect. This intersection is called a
\emph{crossing}.  One can identify a crossing by its
\emph{column} $t$.  We call $a_t$ the \emph{row} of the
crossing at column $t$.  When the word $\mathbf{a}$ is
reduced, the minimality of the length of $\mathbf{a}$
ensures that any two wires cross at most once.  In this
case, we can also identify a crossing by the unordered pair
of wire labels that are involved, i.e.\ the pair
$\{\pi_t(a_t),\pi_t(a_{t+1})\}$.

Note that the terms row and column have slightly different
meaning when we refer to a crossing versus a wire.  The
upper left corner of a wiring diagram is at $(1,0)$.  When
we say a crossing in row $i$ column $j$ it means the
intersection of the crossing is at
$(i+\tfrac12,j-\tfrac12)$.  When we say wire $r$ is in row
$i$ at column $j$, we mean that $\pi_j(i)=r$, so that the
$r$-wire passes through the point $(i,j)$.

Observe that for $i<j$, wires $\pi(i)$ and $\pi(j)$ cross
in the wiring diagram for $\bfa \in R(\pi)$ if and only if
$\pi(i)>\pi(j)$.  This occurs if and only if $(i,j)$ is an
inversion of $\pi$, which in turn is equivalent to the wire
labels $(\pi(j),\pi(i))$ being an inversion of $\pi^{-1}$.
Many of the arguments below depend on the positions of the
inversions for $\pi$ not for $\pi^{-1}$. Reversing any word
for $\pi$ gives a word for $\pi^{-1}$. Thus, if we label
the wires $1,2,3,\ldots$ in increasing order down the right
side of a wiring diagram instead of the left, then the
corresponding wires travel right to left, and appear in the
order $\pi^{-1}$ down the left side. Thus, the $i$-wire and
the $j$-wire cross in the \emph{right-labeled wiring
diagram} for $\bfa \in R(w)$ if and only if $(i,j)$ is an
inversion of $\pi$.

The wiring diagrams shown on the
first
row of
Figure~\ref{fig:bijection} are all right-labeled wiring diagrams.  For
example, the word $(1,3,2)$ corresponding to the second wiring diagram
from the left is a reduced word for the permutation
$[2,4,1,3]=[3,1,4,2]^{-1}$.

\subsection{Bounded bumping algorithm}\label{ss:little algorithm}

Little's bumping algorithm~\cite{little2003combinatorial},
also known as a ``Little bump'',
is a map on reduced words.  It was introduced to study the
decomposition of Stanley symmetric functions into Schur
functions in a bijective way.  Later, the Little algorithm
was found to be related to the Robinson-Schensted-Knuth
map~\cite{little2} and the Edelman-Greene
map~\cite{hamaker-young}; it has been extended to signed
permutations~\cite{billey-hamaker-roberts-young},  affine
permutations \cite{lam-shimozono}, and the subset of involutions in $S_{n}$ \cite{Hamaker-Marberg-Pawlowski.2016}.  The key building block
of our bijective proofs is an enhancement of Little's
algorithm which we call the bounded bumping algorithm. We
describe it below, after setting up notation.


\begin{defn}
    Let $\mathbf{a} = (a_1, \ldots, a_k)$ be a word.  Define the
\emph{decrement-push}, \emph{increment-push},
\emph{deletion} and \emph{insertion} of $\mathbf{a}$ at
column $t$, respectively, to be \begin{align*}
        \Push^-_t \bfa &= (a_1, \ldots, a_{t-1}, a_t-1, a_{t+1}, \ldots, a_k); \\
        \Push^+_t \bfa &= (a_1, \ldots, a_{t-1}, a_t+1, a_{t+1}, \ldots, a_k); \\
        \Delete_t \bfa &= (a_1, \ldots, a_{t-1},  a_{t+1}, \ldots,
        a_k); \\
        \Insert_t^{x} \bfa &= (a_1, \ldots, a_{t-1},x,  a_{t}, \ldots, a_k).
    \end{align*}
\end{defn}
	
In~\cite{BY2014}, the notation $\Push^{\uparrow}$ was used to
represent $\Push^-$, and $\Push^{\downarrow}$ was used to represent
$\Push^+$ based on the direction of a crossing in the wiring diagram.

\begin{defn}
    Let $\mathbf{a}$ be a word.  If $\Delete_t \bfa$ is reduced, then we say that $\mathbf{a}$ is \emph{nearly reduced at $t$}.
\end{defn}

The term ``nearly reduced'' was coined by Lam~\cite[Chapter
  3]{LLMSSZ}, who uses ``$t$-marked nearly reduced''.  Words that are
nearly reduced at $t$ may or may not also be reduced;
however, every reduced word $\bfa$ is nearly reduced at
some index $t$.  For instance, a reduced word $\bfa$ of
length $k$ is nearly reduced at $1$ and at $k$.

In order to define
the bounded bumping algorithm,
we need the following lemma, which
to our knowledge first appeared in~\cite[Lemma
4]{little2003combinatorial}, and was later generalized to arbitrary
Coxeter systems by Lam and Shimozono using the strong exchange
property.  The statement can also be checked for permutations by considering the wiring
diagram.
\begin{lem}\cite[Lemma 21]{lam-shimozono}
\label{lem:nearly_reduced} If $\bfa$ is not reduced, but is
nearly reduced at $t$, then $\bfa$ is nearly reduced at
exactly one other column $t' \neq t$.   In the wiring
diagram of $\bfa$, the two wires crossing in column $t$
cross in exactly one other column $t'$.
\end{lem}

\begin{defn}
In the situation of Lemma~\ref{lem:nearly_reduced}, we say
that $t'$ \emph{forms a defect with} $t$ in $\bfa$, and
write $\defect_t(\bfa) = t'$.
\end{defn}

\begin{figure}
\centering
\includegraphics[width=1.65in]{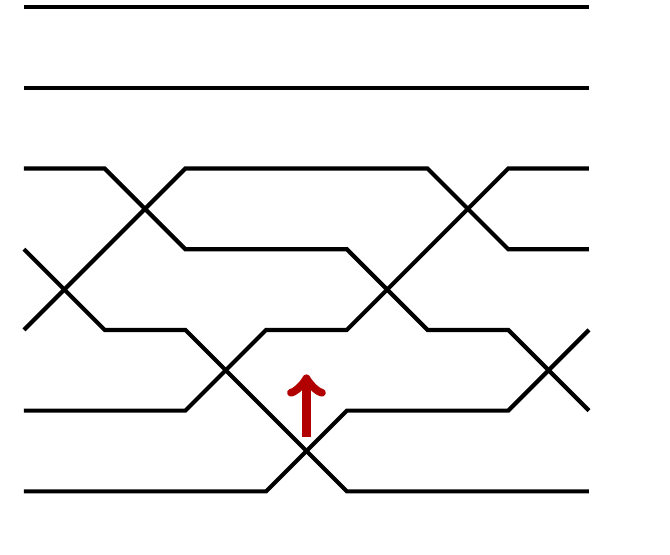}
\includegraphics[width=1.65in]{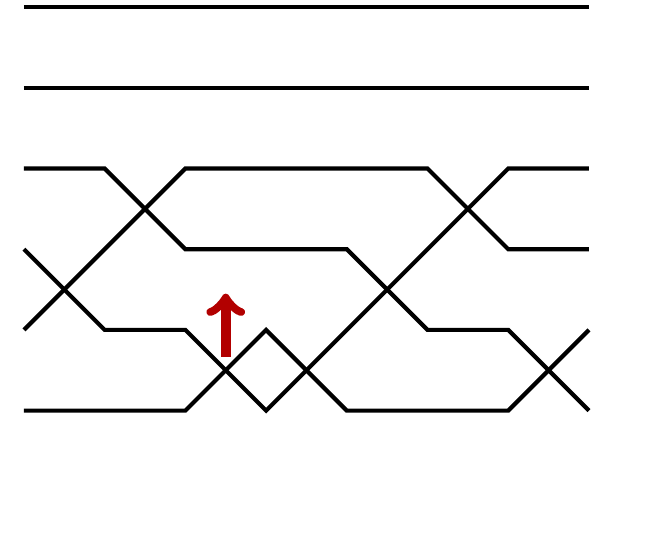}
\includegraphics[width=1.65in]{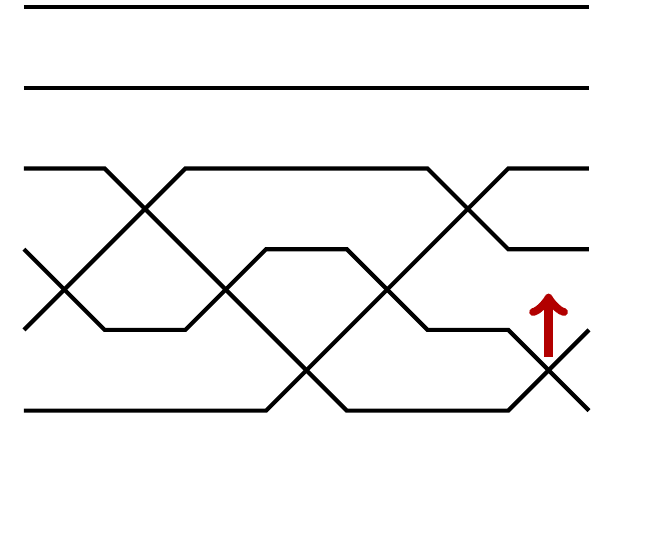} \\
\vspace{.2in}
\includegraphics[width=1.65in]{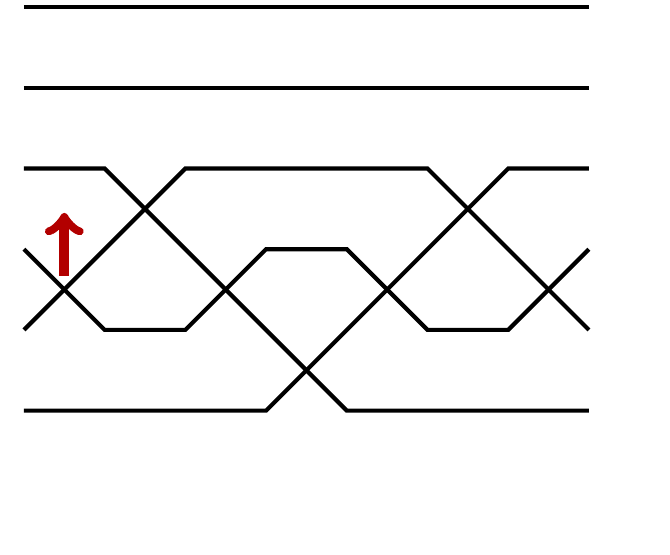}
\includegraphics[width=1.65in]{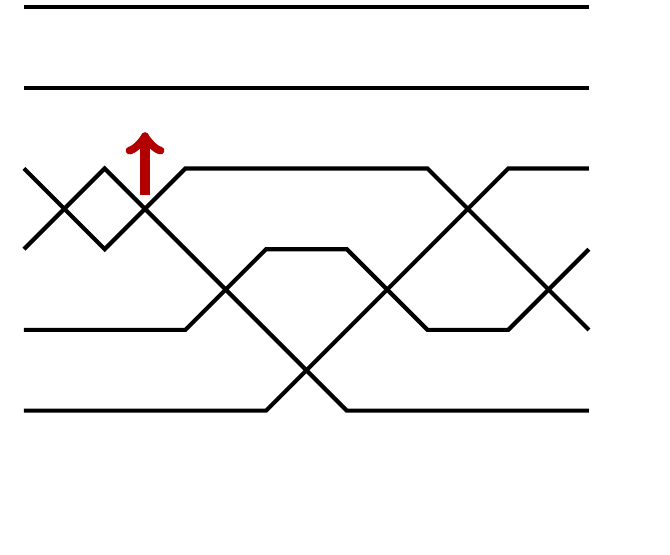}
\includegraphics[width=1.65in]{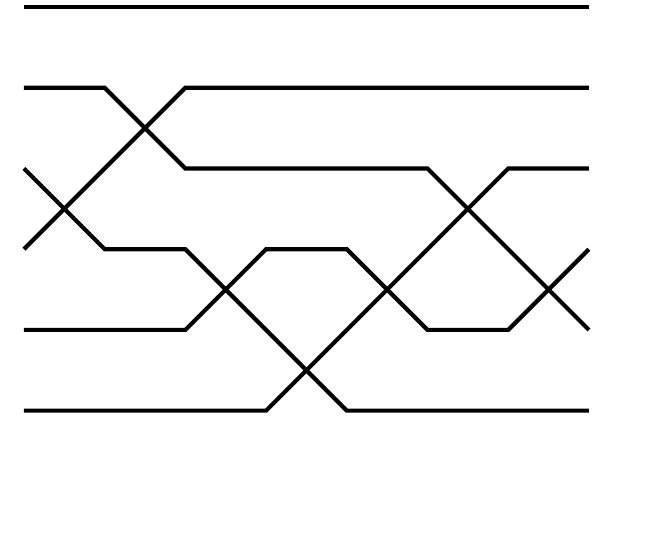}
\caption{An example of the sequence of wiring diagrams for the words
  $\bfa'$ which appear when running the bounded bumping algorithm on
  input
  $\bfa =(4, 3, 5, 6, 4, 3, 5), \ \bfb = (2,2,2,2,2,2,2),\ t_0=4,
  \text{ and } \d=-.$ The arrows indicate which crossing will move in
  the next step.  After the first step, row $7$ contains a wire with
  no swaps, which is therefore not shown.}
\label{fig:little.bump}
\end{figure}

A crucial point is that the definitions of ``reduced'',
``nearly reduced'', and the $\defect$ map make sense even
if we are given only the word $\bfa$, but not the
corresponding permutation $\pi\in S_n$, nor even its size
$n$. Indeed, we can take $n$ to be any integer greater than
the largest element of $\bfa$; it is easily seen that the
three notions coincide for all such $n$.  An alternative,
equivalent viewpoint is to interpret all our permutations
as permutations of $\Z^+:=\{1,2,\ldots\}$ that fix all but
finitely many elements; we can abbreviate such a
permutation $\pi=[\pi(1),\pi(2),\ldots]$ to
$\pi=[\pi(1),\ldots,\pi(n)]$ where $n$ is any integer such
that all elements greater than $n$ are fixed.  Let
$S_{\infty}$ be the set of all such permutations on
$\Z^{+}$.

Our central tool is a modification of the bumping algorithm introduced
by Little in~\cite{little2003combinatorial}.  We call our modified
version the \emph{bounded} bumping algorithm. This algorithm will be
used twice in the proof of Theorem~\ref{t:macdonald}, in two different
contexts.

\begin{defn}
A word $\bfb$ is a \emph{bounded word} for another word
$\bfa$ if the words have the same length and $1\leq b_i\leq
a_i$ for all $i$. A \emph{bounded pair} (for a permutation
$\pi$) is an ordered pair $(\bfa,\bfb)$ such that $\bfa$ is
a reduced word (for $\pi$) and $\bfb$ is a bounded word for
$\bfa$.  Let $\BoundedPairs(\pi)$ be the set of all bounded
pairs for $\pi$.
\end{defn}

\noindent For example, for the simple transposition $s_k$,
the set is
$$\BoundedPairs(s_k) = \bigl\{\bigl((k),(i)\bigr): 1\leq i \leq k\bigr\}.$$

\begin{algorithm}[Bounded Bumping Algorithm]
  \label{algorithm:little bump}\

\noindent \textbf{Input}: $(\bfa, \bfb, t_0, \d)$, where
$\bfa$ is a word that is nearly reduced at $t_0$, and
$\bfb$ is a bounded word for $\bfa$, and $\d \in \{-,
+\}=\{-1,+1 \}$ is a direction.

\noindent \textbf{Output}: $\Bump^{
\d}_{t_0}(\bfa, \bfb) =
(\bfa', \bfb', i, j, \outcome)$, where $\bfa'$ is a reduced
word, $\bfb'$ is a bounded word for $\bfa'$, $i$ is the row
and $j$ is the column of the last crossing pushed in the
algorithm, and $\outcome$ is a binary indicator explained
below.

\begin{enumerate}
    \item Initialize $\bfa'\leftarrow \bfa,\, \bfb'
        \leftarrow \bfb, \, t \leftarrow t_0$.
\item Push in direction
$\d$ at column $t$, i.e.\ set $\bfa' \leftarrow
    \Push^{\d}_{t}\bfa'$ and $\bfb' \leftarrow
    \Push^{\d}_{t}\bfb'$.
\item If $b'_t = 0$, return $(\Delete_t \bfa', \Delete_t
    \mathbf{b}', \bfa'_{t}, t, \deleted)$ and
    \textbf{stop}.
\item If $\bfa'$ is reduced, return $(\bfa', \bfb',
    \bfa'_{t}, t, \bumped)$ and \textbf{stop}.
\item Set $t \leftarrow \defect_t(\bfa')$ and
    \textbf{return to step 2.}
\end{enumerate}
\end{algorithm}

The principal difference between the above algorithm and Little's map
$\theta_r$ in~\cite{little2003combinatorial} is the presence of the
bounded word $\bfb$, which indicates the number of times each column
is allowed to be decremented before being deleted.  The stopping rule
in step 3 is not present in Little's algorithm.  As discussed above,
one consequence is that when $\d=-$, our map can never output a word
containing a $0$: if a push results in a $0$ then it is immediately
deleted and the algorithm stops.  In contrast, in Little's original
algorithm, the entire word is instead shifted by $+1$ in this
situation, changing the permutation (and also stopping, since the word
is reduced).  Indeed, Little's bumping algorithm in the $+1$ direction
on a reduced word $\bfa$ maps to $\bfa'$ if and only if \[
\Bump^{+}_j(\mathbf{a}, \mathbf{b}) = (\mathbf{a}',\mathbf{b}',
i,j,
\bumped) \] regardless of the choice of bounded word $\bfb$ for $\bfa$.

Since this algorithm is the main tool used in the paper, we will give
several examples. In Figure~\ref{fig:little.bump}, we show the
sequence of wiring diagrams for the words $\bfa'$ in the algorithm when it is run on the input
$$\bfa =(4, 3, 5, 6, 4, 3, 5), \ \bfb = (2,2,2,2,2,2,2),\ t_0=4, \text{ and } \d=-.$$
The result is
$$\Bump^{-}_{4}(\bfa, \bfb) =
\bigl((3, 2, 4, 5, 4, 3, 4), (1,1,1,1,2,2,1),2,2, \bumped\bigl).$$
Note that Little's bumping algorithm maps $(4, 3, 5, 6, 4, 3, 5)$ to
$(3, 2, 4, 5, 4, 3, 4)$ using the exact same sequence of pushes as in
Figure~\ref{fig:little.bump} as expected since $\outcome=\bumped$.
On the other hand, with input
$\widetilde{\mathbf{b}} =
(2,2,2,2,2,2,1)$
the bounded bumping algorithm stops after the third push
in the sequence because $\widetilde{b}_7=1$, so
$$\Bump^{-}_{4}(\bfa, \widetilde{\mathbf{b}}) =
\bigl((4, 3, 4, 5, 4, 3), (2,2,1,1,2,2), 4, 7, \deleted\bigr).$$
Another good example for the reader to consider is when the input word $\bfa$  is a consecutive sequence such as
$$\Bump^{-}_{1}\bigl((6,5,4,3), (3,3,3,3)\bigr)=\bigl((5,4,3,2),
(2,2,2,2),2,4,\bumped\bigr).$$

We now make some remarks about this algorithm. The initial
input word $\bfa$ may or may not be reduced, but, if we
reach step $5$ then $\bfa'$ is always not reduced but
nearly reduced at $t$, so the $\defect$ map makes sense.

Suppose that the input word $\bfa$ is a word for a permutation $\pi\in
S_n$. Pushes may in general result in words with elements
outside the interval $[1,n-1]$. Specifically, in the case
$\d=+$, step 2 may result in a word $\bfa'$ with an element
$a_t'=n$. As mentioned above, this can be interpreted as a
word for a permutation in $S_{n+1}$.  In fact, in this case
the algorithm will immediately stop at step 4, since this
new word is necessarily reduced.  On the other hand, in the
case $\d=-$, if step 2 ever results in a word with $a_t'=0$,
we must have $b_t'=0$ as well, so the algorithm will
immediately stop at step 3, and the $0$ will be deleted.
Note that it is also possible for a non-zero element of
$\bfa$ to be deleted at step 3, since $b_{i}'<a_{i}'$ is possible.
Thus, the bounded bumping algorithm clearly
terminates in a finite number of steps.

The proposition below collects several technical facts about the
bounded bumping algorithm that are analogous to facts proved by Little
about his algorithm \cite{little2003combinatorial}. These statements
may be checked by essentially the same arguments as in
\cite{little2003combinatorial} -- the inclusion of $\bfb$ has scant
effect here.

\begin{prop}\label{t:little}
Let $\bfa$ be a word that is nearly reduced at $t$, let $\bfb$ be a
bounded word for $\bfa$, and let $\d \in \{ +, - \}$.  Assume
$\Bump^{\d}_{t}(\bfa, \bfb) = (\bfa',\bfb',i,j,\outcome)$.
\begin{enumerate}
\item Suppose $\bfa$ is reduced. Then, Algorithm~\ref{algorithm:little
bump} is reversible in the sense that we can recover the inputs by
negating the direction $\d$. More specifically, if $\outcome=\deleted$,
then $\d=-1$ and $\Bump^{-\d}_{j}(\Insert_{j}^{i} \bfa',
\Insert_{j}^{0} \bfb') = (\bfa, \bfb, \bfa_{t}, t, \bumped)$; if
$\outcome=\bumped$, then $\Bump^{-\d}_j(\bfa', \bfb') = (\bfa, \bfb,
\bfa_{t}, t, \bumped)$.

\item \label{item:2} If $\bfa \in R(\pi)$, then
          $\Delete_t \bfa \in R(\pi t_{k,l})$, where $(k<l)$ is
    the inversion of $\pi$ whose wires cross in column
    $t$ of the right-labeled wiring diagram for $\bfa$.
    If $\outcome = \bumped$, then  $\bfa' \in R(\pi t_{k,l} t_{x,y})$ where
$\{x<y\}$
is the crossing in column $j$ of the word
    $\bfa'$ for $\pi t_{k,l} t_{x,y}$. Furthermore, if
    $\d=+$, then $l=x$. If $\d=-$, then $k=y$.

\item \label{item:3} 
Suppose $\Delete_j \mathbf{a} \in R(\nu)$. 
 After every iteration of step
    2 in the bounded bumping algorithm computing
    $\Bump^{\d}_{t}(\bfa, \bfb)$, the pair $(\Delete_t
    \bfa', \Delete_t \mathbf{b}')$ is a bounded pair for
    $\nu$.   In particular, if $\outcome = \deleted$, then
    $\bfa' \in R(\nu)$.

\item If $\outcome=\bumped$, then the input and output words $\bfa$
and $\bfa'$ have the same ascent set.  If $\outcome=\deleted$, then
the ascent set of $\Insert_{j}^{i}(\bfa')$ is the same as the ascent
set of $\bfa$.
\end{enumerate}
\end{prop}

Note that in items (3) and (4) above, the word $\bfa$ is not
necessarily reduced.


\subsection{Pipe Dreams and Schubert Polynomials}\label{ss:pipe dreams and schubert polynomials}

Schubert polynomials $\mathfrak{S}_\pi$ for $\pi \in S_n$
are a generalization of Schur polynomials invented by
Lascoux and Sch\"utzenberger in the early 1980s \cite{LS1}.
They have been widely used and studied over the past 30
years.  An excellent summary of the early work on these
polynomials appears in Macdonald's notes \cite{M2}; see
Manivel's  book  \cite{manivel-book} for a more recent
treatment.

A \emph{pipe dream} $D$ is a finite subset of $\mathbb{Z}_+
\times \mathbb{Z}_+$.  We will usually draw a pipe dream as
a modified wiring diagram as follows.   Place a $+$ symbol
at every point $(i,j) \in D$; place a pair of elbows
$\elbows$ at every other point $(i,j) \in (\mathbb{Z}_+
\times \mathbb{Z}_+) \setminus D$, where again we use
matrix-style coordinates. This creates wires connecting
points on the left side of the diagram to points on the
top.  If the wires are numbered $1, 2, 3, \ldots$ down the left side, then the corresponding wires
reading along the top of the diagram from left to right form a
permutation $\pi$ of the positive integers that fixes all
but finitely many values. We call $\pi^{-1}$  the
\emph{permutation of $D$} following the literature.

We call the elements of a pipe dream $D \subset \mathbb{Z}_+ \times
\mathbb{Z}_+ $ \emph{crossings} or \emph{occupied positions}, and the
elements $(i,j)$ of $ \mathbb{Z}_+ \times \mathbb{Z}_+ \setminus D$
\emph{unoccupied positions}.  Each crossing involves two wires, which
are said to \emph{enter the crossing horizontally} and
\emph{vertically}.

Following the terminology for reduced words, we say that $D$ is
\emph{reduced} if $\pi$ is the permutation of $D$ and $\ell(\pi) =
|D|$.  We write $\rp(\pi)$ for the set of all reduced pipe dreams for
$\pi$.  Two wires labeled $i<j$ cross somewhere in $D \in \rp(\pi)$ if
and only if $(i,j)$ is an inversion of $\pi$.  Observe that the
smaller labeled wire necessarily enters the crossing horizontally in a
reduced pipe dream.

As mentioned earlier, we can identify the permutation of a pipe dream with one in
$S_n$, where all elements greater than $n$ are fixed.  We only need to
draw a finite number of wires in a triangular array to represent a
pipe dream since for all large enough wires there are no crossings.
See Figure~\ref{fig:rcgraphs} for an example.

\begin{figure}
\begin{center}
\begin{tikzpicture}[scale=0.3]
\begin{scope}[xshift=40em, scale=0.9, thick]
\begin{scope}[transparency group, opacity=0.75]
\draw[-stealth,line width=5pt, orange!50!yellow] (11, 10)-- (-3, 10);
\draw[-stealth,line width=5pt, orange!50!yellow]  (9, 8) -- (-3, 8);
\draw[-stealth,line width=5pt, orange!50!yellow]  (7, 6) -- (-3, 6);
\draw[-stealth,line width=5pt, orange!50!yellow]  (5, 4) -- (-3, 4);
\draw[-stealth,line width=5pt, orange!50!yellow]  (3, 2) -- (-3, 2);
\draw[-stealth,line width=5pt, orange!50!yellow]  (1, 0) -- (-3, 0);
\end{scope}\draw (-0.5,8) arc (270:360:0.5)
(0,8.5) -- (0,11.5);
\draw (-0.5,0) arc (270:360:0.5)
(0,0.5) -- (0,3.5)
(0,3.5) arc (180:90:0.5)
(0.5,4) -- (1.5,4)
(1.5,4) arc (270:360:0.5)
(2,4.5) -- (2,11.5);
\draw (-0.5,10) -- (3.5,10)
(3.5, 10) arc (270:360:0.5)
(4,10.5) -- (4,11.5);
\draw (-0.5,6) arc (270:360:0.5)
(0,6.5) -- (0,7.5)
(0,7.5)  arc (180:90:0.5)
(0.5,8) -- (3.5,8)
(3.5,8) arc (270:360:0.5)
(4,8.5) -- (4,9.5)
(4,9.5) arc (180:90:0.5)
(4.5,10) -- (5.5,10)
(5.5,10) arc (270:360:0.5)
(6,10.5) -- (6,11.5);
\draw (-0.5,2) -- (1.5,2)
(1.5,2) arc (270:360:0.5)
(2,2.5) -- (2,3.5)
(2,3.5) arc (180:90:0.5)
(2.5,4) -- (3.5,4)
(3.5,4) arc (270:360:0.5)
(4,4.5) -- (4,5.5)
(4,5.5) arc (180:90:0.5)
(4.5,6) -- (5.5,6)
(5.5,6) arc (270:360:0.5)
(6,6.5) -- (6,7.5)
(6,7.5) arc (180:90:0.5)
(6.5,8) -- (7.5,8)
(7.5,8) arc (270:360:0.5)
(8,8.5) -- (8,11.5);
\draw (-0.5,4) arc (270:360:0.5)
(0,4.5) -- (0,5.5)
(0,5.5) arc (180:90:0.5)
(0.5,6) -- (3.5,6)
(3.5,6) arc (270:360:0.5)
(4,6.5) -- (4,7.5)
(4,7.5) arc (180:90:0.5)
(4.5,8) -- (5.5,8)
(5.5,8) arc (270:360:0.5)
(6,8.5) -- (6,9.5)
(6,9.5) arc (180:90:0.5)
(6.5,10) -- (9.5,10)
(9.5,10) arc (270:360:0.5)
(10,10.5) -- (10,11.5);
\draw (8.6,10.9) node {\textbf{1}};
\draw (2.6,10.9) node {\textbf{2}};
\draw (0.6,10.9) node {\textbf{3}};
\draw (2.6, 8.9) node {\textbf{4}};
\draw (2.6, 6.9) node {\textbf{5}};
\draw (0.6, 2.9) node {\textbf{6}};
\end{scope}

\begin{scope}[scale=0.9,thick]
\draw (-0.5,8) arc (270:360:0.5)
(0,8.5) -- (0,11.5);
\draw (-0.5,0) arc (270:360:0.5)
(0,0.5) -- (0,3.5)
(0,3.5) arc (180:90:0.5)
(0.5,4) -- (1.5,4)
(1.5,4) arc (270:360:0.5)
(2,4.5) -- (2,11.5);
\draw (-0.5,10) -- (3.5,10)
(3.5, 10) arc (270:360:0.5)
(4,10.5) -- (4,11.5);
\draw (-0.5,6) arc (270:360:0.5)
(0,6.5) -- (0,7.5)
(0,7.5)  arc (180:90:0.5)
(0.5,8) -- (3.5,8)
(3.5,8) arc (270:360:0.5)
(4,8.5) -- (4,9.5)
(4,9.5) arc (180:90:0.5)
(4.5,10) -- (5.5,10)
(5.5,10) arc (270:360:0.5)
(6,10.5) -- (6,11.5);
\draw (-0.5,2) -- (1.5,2)
(1.5,2) arc (270:360:0.5)
(2,2.5) -- (2,3.5)
(2,3.5) arc (180:90:0.5)
(2.5,4) -- (3.5,4)
(3.5,4) arc (270:360:0.5)
(4,4.5) -- (4,5.5)
(4,5.5) arc (180:90:0.5)
(4.5,6) -- (5.5,6)
(5.5,6) arc (270:360:0.5)
(6,6.5) -- (6,7.5)
(6,7.5) arc (180:90:0.5)
(6.5,8) -- (7.5,8)
(7.5,8) arc (270:360:0.5)
(8,8.5) -- (8,11.5);
\draw (-0.5,4) arc (270:360:0.5)
(0,4.5) -- (0,5.5)
(0,5.5) arc (180:90:0.5)
(0.5,6) -- (3.5,6)
(3.5,6) arc (270:360:0.5)
(4,6.5) -- (4,7.5)
(4,7.5) arc (180:90:0.5)
(4.5,8) -- (5.5,8)
(5.5,8) arc (270:360:0.5)
(6,8.5) -- (6,9.5)
(6,9.5) arc (180:90:0.5)
(6.5,10) -- (9.5,10)
(9.5,10) arc (270:360:0.5)
(10,10.5) -- (10,11.5);
\draw (-1,0) node {\small{6}};
\draw (-1,2) node {\small{5}};
\draw (-1,4) node {\small{4}};
\draw (-1,6) node {\small{3}};
\draw (-1,8) node {\small{2}};
\draw (-1,10) node {\small{1}};
\draw (0,12.3) node {\small{2}};
\draw (2,12.3) node {\small{6}};
\draw (4,12.3) node {\small{1}};
\draw (6,12.3) node {\small{3}};
\draw (8,12.3) node {\small{5}};
\draw (10,12.3) node {\small{4}};
\end{scope}

\begin{scope}[xshift=80em, yshift=2ex, scale=1.75, thick]

\node at (-0.3,5) {$2$}; \node at (-0.3,4) {$6$}; \node at
(-0.3,3) {$1$}; \node at (-0.3,2) {$3$}; \node at (-0.3,1)
{$5$}; \node at (-0.3,0) {$4$};

\node at (6.3,5) {$1$}; \node at (6.3,4) {$2$}; \node at
(6.3,3) {$3$}; \node at (6.3,2) {$4$}; \node at (6.3,1)
{$5$}; \node at (6.3,0) {$6$};

\coordinate (0/6) at (0, 0); \coordinate (0/5) at (0, 1);
\coordinate (0/4) at (0, 2); \coordinate (0/3) at (0, 3);
\coordinate (0/2) at (0, 4); \coordinate (0/1) at (0, 5);

\coordinate (1/5) at (1, 0); \coordinate (1/6) at (1, 1);
\coordinate (1/4) at (1, 2); \coordinate (1/3) at (1, 3);
\coordinate (1/2) at (1, 4); \coordinate (1/1) at (1, 5);

\coordinate (2/5) at (2, 0); \coordinate (2/6) at (2, 1);
\coordinate (2/4) at (2, 2); \coordinate (2/2) at (2, 3);
\coordinate (2/3) at (2, 4); \coordinate (2/1) at (2, 5);

\coordinate (3/5) at (3, 0); \coordinate (3/6) at (3, 1);
\coordinate (3/4) at (3, 2); \coordinate (3/2) at (3, 3);
\coordinate (3/1) at (3, 4); \coordinate (3/3) at (3, 5);

\coordinate (4/5) at (4, 0); \coordinate (4/6) at (4, 1);
\coordinate (4/2) at (4, 2); \coordinate (4/4) at (4, 3);
\coordinate (4/1) at (4, 4); \coordinate (4/3) at (4, 5);

\coordinate (5/5) at (5, 0); \coordinate (5/2) at (5, 1);
\coordinate (5/6) at (5, 2); \coordinate (5/4) at (5, 3);
\coordinate (5/1) at (5, 4); \coordinate (5/3) at (5, 5);

\coordinate (6/2) at (6, 0); \coordinate (6/5) at (6, 1);
\coordinate (6/6) at (6, 2); \coordinate (6/4) at (6, 3);
\coordinate (6/1) at (6, 4); \coordinate (6/3) at (6, 5);

\draw (0/1) -- (1/1) -- (2/1) -- (3/1) -- (4/1) -- (5/1) --
(6/1); \draw (0/2) -- (1/2) -- (2/2) -- (3/2) -- (4/2) --
(5/2) -- (6/2); \draw (0/3) -- (1/3) -- (2/3) -- (3/3) --
(4/3) -- (5/3) -- (6/3); \draw (0/4) -- (1/4) -- (2/4) --
(3/4) -- (4/4) -- (5/4) -- (6/4); \draw (0/5) -- (1/5) --
(2/5) -- (3/5) -- (4/5) -- (5/5) -- (6/5); \draw (0/6) --
(1/6) -- (2/6) -- (3/6) -- (4/6) -- (5/6) -- (6/6);

\end{scope}
\end{tikzpicture}
 \caption{Left: a reduced pipe dream $D$ for
   $\pi=[3,1,4,6,5,2]=[2,6,1,3,5,4]^{-1}$.
The weight is $x^D=x_1^3 x_2 x_3
   x_5$.
Middle: the reading
   order for the crossings, with numbers indicating
   position in the order.  The resulting sequences of row
   numbers and column numbers are $\mathbf{i}_D=(1,1,1,2,3,5)$ and
   $\mathbf{j}_D=(5,2,1,2,2,1)$ respectively.  Right: the right-labeled
   wiring diagram of the
   associated reduced word
   $\mathbf{r}_D =(5,2,1,3,4,5) \in R(\pi)$.
   }
 \label{fig:rcgraphs}
\end{center}
\end{figure}
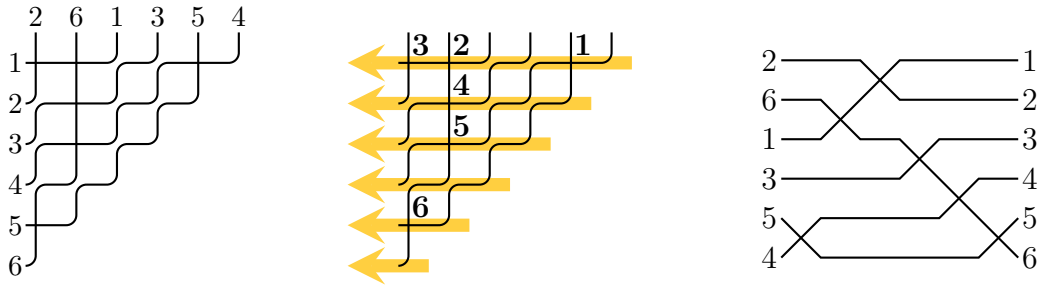

The \emph{weight} of a pipe dream $D$ is given by the product over row
numbers of the crossings
\[
 x^D := \prod_{(i,j) \in D} x_i
\]
where $x_1, x_2, \ldots$ are formal variables.  The Schubert
polynomial can be defined as a generating function for weighted
reduced pipe dreams as follows.

\begin{defn}\label{d:schubs}
The \emph{Schubert polynomial} of $\pi \in S_n$ is defined to be
\[
    \mathfrak{S}_{\pi}=\mathfrak{S}_{\pi}(x_1,x_2,\ldots,x_n) := \sum_{D \in \rp(\pi)} x^D.
\]
\end{defn}

For example, the second row of Figure~\ref{fig:bijection} shows pipe
dreams for 5 different permutations.  The pipe dream in the middle of
the figure is the unique reduced pipe dream for  $[2,3,1,4]=[3,1,2,4]^{-1}$  so
$\mathfrak{S}_{[2,3,1,4]}=x_1x_{2}$.  The pipe dream on the left for
$[1,4,3,2]$ is not the only one.  There are 5 pipe dreams for
$w=[1,4,3,2]$ in total and

\begin{equation}\label{ex:1}
\mathfrak{S}_{[1,4,3,2]} =  x_1^2 x_2 + x_1^2x_3 + x_1x_2^2 + x_1x_2x_3 + x_2^2x_3.
\end{equation}

There are many other equivalent definitions of Schubert polynomials~
\cite{billey-bergeron,BJS,FK,Fomin-Stanley,LS1,Weigandt-Yong.2015}.
Note that pipe dreams are also called \emph{pseudo-line} arrangements
and \emph{RC-graphs} in the literature.  See
\cite{knutson-miller-2005,Kogan.phd} for other geometric and algebraic
interpretations of individual pipe dreams.

The following theorem is an important tool for calculating Schubert
polynomials.  It is a recurrence based on the lexicographically
(lex)
largest inversion $(r,s)$ for $\pi$ assuming $\pi \neq
\id$, where as usual an inversion means $r<s$ and
$\pi(r)>\pi(s)$.  Note that $r$ is the position of the
largest descent in $\pi$, and $s$ is the largest value such
that $\pi(r) > \pi(s)$.  If $\bfa$ is a reduced word for
$\pi$, then there exists a unique column $t_0$ containing
the $\{r,s\}$-wire crossing in the right-labeled wiring
diagram for $\bfa$.  One can easily verify that $\ell(\pi
t_{rs})=\ell(\pi) -1$, and hence $\bfa$ is nearly reduced
in column $t_0$.  The original proof due to Lascoux and
Sch\"utzenberger \cite{LS2} uses Monk's formula for
computing products of Schubert classes in the cohomology
ring of the flag manifold.  See also \cite[4.16]{M2}. We
give a bijective proof using pipe dreams in the next
section.

\begin{samepage}
\begin{thm}[Transition Equation for Schubert polynomials;
\cite{LS2}] \label{t:transitionA}
For all permutations $\pi$ with $\pi \neq \id$, the
Schubert polynomial $\mathfrak{S}_{\pi}$ is determined by
the recurrence
\begin{equation}\label{e:transA}
\mathfrak{S}_{\pi}= x_{r}\mathfrak{S}_{\nu} + \sum_{\substack{
q<r:
\\
\ell(\pi)=\ell(\nu t_{qr})
}}
\mathfrak{S}_{\nu t_{qr}}
\end{equation}
where $(r,s)$ is the
lex
largest inversion in
$\pi$, and $\nu=\pi t_{rs}$.  The base case of the
recurrence is $\mathfrak{S}_{\id}=1$.
\end{thm}
\end{samepage}

Continuing the example above, the lex largest inversion for
$w=[1,4,3,2]$ is $(3,4)$ so
$$
\mathfrak{S}_{[1,4,3,2]} =  x_3 \mathfrak{S}_{[1,4,2,3]} + \mathfrak{S}_{[2,4,1,3]}.
$$
The lex largest  inversion for $[2,4,1,3]$ is $(2,4)$ so
$$
\mathfrak{S}_{[2,4,1,3]} =  x_2 \mathfrak{S}_{[2,3,1,4]} + \mathfrak{S}_{[3,2,1,4]}.
$$ If we continue to use the Transition Equation, we find
$\mathfrak{S}_{[3,2,1,4]} = x^2_1 x_2$, $\mathfrak{S}_{[2,3,1,4]}= x_1
x_2$ and $\mathfrak{S}_{[1,4,2,3]}= x_2 \mathfrak{S}_{[1,3,2,4]}+
\mathfrak{S}_{[3,1,2,4]}$ $= x_2( x_1 + x_2) +x_1^2$.  Therefore, we
can rederive \eqref{ex:1} via the Transition Equation as well.

\begin{defn}\label{defn:lex.largest.inversion.order}
We define the \emph{inversion order} $\lessthan$ on
permutations as follows. Given $\pi \in S_{\infty}$, let
$\Inv(\pi)$ be the ordered list of inversions in reverse
lex order.  Note, $\Inv(\pi)$ begins with the lex largest
inversion of $u$.  Then, for $\tau,\pi \in S_{\infty}$, we
say $\tau \lessthan \pi $ provided $\Inv(\tau ) < \Inv(\pi)
$ in lex order as lists.  For example,
$\Inv([1432])=((3,4),(2,4),(2,3))$ and
$\Inv([2413])=((2,4),(2,3),(1,3))$, so $[2413]\prec[1432]$.
\end{defn}

\begin{remark}\label{rem:lex.largest.inversion.order}
  All of the permutations on the right hand side of \eqref{e:transA}
are strictly smaller than $\pi$ in inversion order by construction.
Furthermore, the permutations on the right hand side of
\eqref{e:transA} are in $S_{n}$ provided $\pi \in S_{n}$. Hence there
are only a finite number of terms in the expansion of a Schubert
polynomial.  We will apply induction over this finite set in the
bijective proofs that follow.
\end{remark}

\section{Bijective proof of the Transition Equation}\label{s:trans}

In this section, we give a bijective proof of the Transition Equation
for Schubert polynomials, Theorem~\ref{t:transitionA}.  The Transition
Algorithm described here is a key tool for proving
Theorem~\ref{t:macdonald}.  We begin by describing how the bounded
bumping algorithm acts on reduced pipe dreams.

A pipe dream for a permutation $\pi$ may be interpreted as a bounded
pair of a special type for the same $\pi$.  To make this more precise,
order the crossings in $D$ in the order given by reading rows from top
to bottom, and from right to left within each row.  We call this the
\emph{reading order} on $D$.   We construct three words from the ordered list
of crossings: the row numbers of the crossings $\mathbf{i}_{D} =
(i_{1}, i_{2},\ldots, i_{p})$, the column numbers $\mathbf{j}_{D} =
(j_{1}, j_{2},\ldots, j_{p})$ and the diagonal numbers $\mathbf{r}_{D}
= (i_{1}+j_{1}-1,i_{2}+j_{2}-1,\ldots, i_{p}+j_{p}-1) =
\mathbf{j}_D+\mathbf{i}_D -\mathbf{1}$.  Any two of
$\mathbf{i}_D,\mathbf{j}_D,\mathbf{r}_D$ suffice to determine $D$.  In
this paper, we will encode $D$ by the \emph{biword}
$(\mathbf{r}_{D},\mathbf{j}_{D})$ departing from the literature which
typically uses $(\mathbf{r}_{D},\mathbf{i}_{D})$.

If $D$ is a pipe dream for $\pi$ then it is easy to see that
$\mathbf{r}_{D}$ is a word for $\pi$.  And $D$ is reduced if and only
if $\mathbf{r}_{D}$ is.  Furthermore, the column numbers
$\mathbf{j}_{D}$ always form a bounded word for $\mathbf{r}_{D}$.
Although the bounded pair $(\mathbf{r}_{D}, \mathbf{j}_{D})$
determines $D$, not every bounded pair corresponds to a pipe dream.
In fact, a bounded pair $(\mathbf{a}, \mathbf{b})=((a_{1},\ldots ,
a_{p}), (b_{1},\ldots, b_{p}))$ corresponds to a pipe dream if and
only if the list $[(i_1,b_1),\ldots,(i_p,b_p)]$ has $p$ distinct
elements  listed in the reading order, where $i_k=a_k-b_k+1$.
Equivalently, $(\mathbf{a}, \mathbf{b})$ corresponds to a pipe dream
if and only if the pairs $(i_1,-b_1),\ldots,(i_p,-b_p)$ are in
strictly increasing lex order.

For example, Figure~\ref{fig:rcgraphs} shows the pipe dream $D$
corresponding with reduced word $\mathbf{r}_{D}=(5,2,1,3,4,5)$, row
numbers $\mathbf{i}_{D}=(1,1,1,2,3,5)$, and column numbers
$\mathbf{j}_{D}=(5,2,1,2,2,1)$.

Using the biword $(\mathbf{r}_{D},\mathbf{j}_{D})$ to encode a reduced
pipe dream $D$, we can apply the bounded bumping algorithm to $D$ in
either direction and for any $t_0$ where $\mathbf{r}_D$ is nearly
reduced.  One can observe that the bounded pairs encountered during
the steps of the bounded bumping algorithm do \emph{not} all encode
pipe dreams, but it will turn out that the departures from ``pipe
dream encoding status'' are temporary, and have a straightforward
structure that will be analyzed in the proof of
Lemma~\ref{l:stack.push} below.

\begin{lem}\label{l:stack.push}
Let $D$ be a reduced pipe dream and suppose that $r_{D}$ is
nearly reduced at $t$.  Let $\d\in \{+,- \}$ and write
$$\Bump^{\d}_{t}(\mathbf{r}_{D}, \mathbf{j}_{D})=
(\bfa',\bfb', i, j, \outcome).$$ Then the bounded pair
$(\bfa',\bfb')$ also encodes a reduced pipe dream.
\end{lem}

\begin{proof}
Consider the effect of the bounded bumping algorithm in terms of pipe
dreams.  To be concrete, assume $\d=-$, the case $\d=+$ being similar.
Observe that when we initially decrement-push
$(\mathbf{r}_{D},\mathbf{j}_{D}) $ in column $t$, it has the effect of
moving the $t^{th}$ crossing in the reading order on $D$, say in
position $(i,j)\in D$, one column to the left to position
$(i,j-1)$. If this location is already occupied, $(i,j-1) \in D$, then
$\Push_{t}^{-} r_{D}$ returns a nearly reduced word with identical
letters in positions $t$ and $t+1$.  The resulting bounded pair does
not encode a pipe dream.  Then,
the
next step of the bounded bumping
algorithm will decrement-push at $t+1$. If $(i,j-2) \in D$ also, then
$\bfa'= \Push_{t+1}^{-}\Push_{t}^{-} r_{D}$ will again have duplicate
copies of the letter $i+j-1$ in positions $t+1$ and $t+2$ so the next
decrement-push will be in position $t+2$, and so on.  Note that since
the algorithm decrement-pushes both of the words in the bounded pair
in the same position at each iteration, the entrywise differences
$\bfa'-\bfb' = \mathbf{r}_{D} - \mathbf{j}_{D}$
agree, so the
original row numbers $\mathbf{i}_{D}$
are maintained
 unless a deletion occurs.

We can group the push steps along one row so a decrement-push in
position $(i,j)$ pushes all of the adjacent $+$'s to its left over by
one as a stack.  Thus, the effect of the bounded bumping algorithm on
the pipe dream amounts to a sequence of such ``stack pushes''.  If at
the end of a stack push, a $+$ in column 1 of the pipe dream is
decrement-pushed, the bounded bump algorithm terminates by deleting
that position because there will be a 0 in the bounded word.
Otherwise, a stack push ends with a bounded pair that corresponds to a
pipe dream, which may or may not be reduced. If it is reduced, the
algorithm stops and returns $\outcome=\bumped$.  Otherwise, we find
the defect and continue with another stack push in a different row.
In either case, the final bounded pair $(\bfa',\bfb')$ encodes a reduced pipe dream.
\end{proof}

Next we give the promised bijective proof of the Transition Equation
for Schubert polynomials using pipe dreams.  The bijection we give was
independently observed by Anders Buch \cite[p.11]{Knutson.2012}.  The
proof will involve several technical steps,
Lemmas~\ref{l:deleted.step} to \ref{l:bumped.step.+}, which are stated
and proved after the main argument.

\begin{proof}[Proof of Theorem~\ref{t:transitionA}]
In the case $\pi = \id$, we have $\mathfrak{S}_{\pi}=1$ so the theorem holds
trivially.  Assume $\pi \neq \id$.  Recall $\nu = \pi t_{r,s}$, and let
\begin{equation}\label{eq:Tset}
  \Tset(\pi):= \rp(\nu) \cup \bigcup_{\substack{ q<r:
      \\
      \ell(\pi)=\ell(\nu t_{qr}) }}
  \rp(\nu t_{qr}).
\end{equation}
We think of $\nu = \nu t_{r,r}$ so each pipe dream in $\Tset(\pi)$ is
for a permutation of the form $\nu t_{q,r}$ with $1\leq q\leq r$,
though not all such $\nu t_{q,r}$ necessarily occur.

By definition, the left side of \eqref{e:transA} is the sum of $x^D$
over all $D\in \rp(\pi)$.  Similarly, the right side can be expressed
as a sum over all reduced pipe dreams $E \in \Tset(\pi)$.  Each such
$E$ contributes either $x_{r} x^E$ or $x^E$ respectively to the sum on
the right side.  We will give a bijection $\TransitionMap: \rp(\pi)
\longrightarrow \Tset(\pi)$ that preserves weight, except in the cases
$\TransitionMap(D)=E \in \rp(\nu)$, where the weight will change by $x_r$, so $x^D
= x_r x^E$.

\begin{algorithm}[Transition Map] \label{algorithm:TransitionMap}\
Suppose $\pi \neq \id$ is given, and let $(r,s)$ and $\nu$
be defined as in Theorem~\ref{t:transitionA}.

 \noindent \textbf{Input}: $D$, a non-empty reduced pipe dream for
 $\pi$ encoded as the biword $(\mathbf{r}_{D},\mathbf{j}_{D})$.

\noindent \textbf{Output}: $\TransitionMap(D) = E \in \Tset(\pi)$.

\begin{enumerate}

\item  Let $t_0$ be the unique column containing the $\{r,s\}$-wiring
   crossing in the right-labeled wiring diagram for $\mathbf{r}_D$.

\item Compute $\Bump^{-}_{t_0}(\mathbf{r}_{D}, \mathbf{j}_{D}) =
  (\bfa',\bfb', i, j, \outcome).$

\item If $\outcome=\deleted$, then we will show in
  Lemma~\ref{l:deleted.step} that $i=r-1$, $j=\ell(\pi)$, and
  $(\mathbf{a}', \bfb')$ encodes a pipe dream $E \in\rp(\nu) \subset
  \Tset(\pi)$.  Return $E$ and \textbf{stop}.

  \item If $\outcome=\bumped$, then we will show in
   Lemma~\ref{l:bumped.step} that $(\mathbf{a}', \bfb')$ encodes a
    pipe dream $E \in\rp(\nu t_{qr})$ for some $q<r$ with
    $\ell(\pi)=\ell(\nu t_{qr})$.  Thus, $E \in \Tset(\pi)$.  Return
    $E$ and \textbf{stop}.
\end{enumerate}
 \end{algorithm}

See Example~\ref{ex:T} below.  The inverse map
$\TransitionMap^{-1}(E)$ again has  two cases.

\begin{algorithm}[Inverse Transition Map] \label{algorithm:invTransitionMap}\
Suppose $\pi \neq \id$ is given, and let $(r,s)$ and $\nu$
be defined as in Theorem~\ref{t:transitionA}.

   \noindent \textbf{Input}: $E \in \Tset(\pi)$ a reduced pipe dream
   encoded by the biword $(\mathbf{r}_{E},\mathbf{j}_{E})$. In particular
   $E \in \rp(\nu t_{q,r})$ for some $1 \leq q \leq r$.

 \noindent \textbf{Output}: $\TransitionMap^{-1}(E) = D$, a reduced
 pipe dream for $\pi$.

\begin{enumerate}

  \item If $q=r$, then $E \in \rp(\nu )$.  Set $j \leftarrow
    \ell(\pi)$, $\mathbf{g} \leftarrow \Insert_{j}^{r-1}(
    \mathbf{r}_{E}),\,$ $\mathbf{h} \leftarrow \Insert_{j}^{0}
    (\mathbf{j}_{E}).$

\item If $q<r$, set $\mathbf{g} \leftarrow \mathbf{r}_{E}$ and $\mathbf{h}
\leftarrow \mathbf{j}_{E}$.  Let $j$ be the column containing the
$\{q,r\}$-crossing in the right-labeled wiring diagram of
$\mathbf{r}_{E}$ which must exist since $E \in \rp(\nu t_{q,r})
\subset \Tset(\pi)$.

\item Compute $\Bump^{+}_j(\mathbf{g}, \mathbf{h}) =
  (\mathbf{g}',\mathbf{h}', i', t, \bumped)$.  Here the outcome will
  always be $\bumped$ since we are applying
  increment-pushes. Lemma~\ref{l:bumped.step.+} below shows that
  $(\mathbf{g}',\mathbf{h}')$ encodes a pipe dream $D \in \rp(\pi)$.

\item Return $D$ and \textbf{stop}.
\end{enumerate}
\end{algorithm}

We claim that $\TransitionMap$ is an injection.  First note that, by
Proposition~\ref{t:little}(1), the bounded bumping algorithm is
reversible given the column $j$ of the final push and, in addition, in
the case the outcome is $\deleted$, the value $i$ of the letter
omitted.  Thus, to prove injectivity, let $E \in \Tset(\pi)$, then $E
\in \rp(\nu t_{q,r})$ for some $q \leq r$.  We need to show $i$ and
$j$ can be determined from $q$.

If $q=r$, then the final push in the bounded bumping algorithm was in
the last position so $i=r-1$ and $j=\ell(\pi)$. If $q<r$ then
$(q,r)$ is an inversion in $\pi t_{r,s} t_{q,r}$.
Lemma~\ref{l:bumped.step} shows that $j$ is determined by the unique
column of the right-labeled wiring diagram of $r_E$ containing the
$\{q,r\}$-wire crossing which must exist since it corresponds to an
inversion.  Thus, $\TransitionMap^{-1}\TransitionMap(D) = D$ so if
$\TransitionMap(D)=E=\TransitionMap(D')$ then $D=D'$.

Similarly, for all $E \in \Tset(\pi)$ we have $\TransitionMap
\TransitionMap^{-1}(E) = E$ so $\TransitionMap$ is surjective.
Therefore, $\TransitionMap$ is a bijection.

Finally, we show that $\TransitionMap$ is weight preserving.  Say
$D\in \rp(\pi)$ and $\TransitionMap(D)=(E,(q,r))$.  Recall that if the
row numbers $\mathbf{i}_D = (i_1,\ldots,i_p)$, then $x^D=x_{i_1}\cdots
x_{i_p}$.  The row numbers are determined by $\mathbf{i}_D =
\mathbf{r}_D - \mathbf{j}_D + \mathbf{1}$ so they are preserved by
each push step in the bounded bumping algorithm since $r_k - j_k$ is
preserved for each $k$.  If $q<r$, then the algorithm terminates with
$x^D=x^E$.  If $q=r$ then the algorithm terminates when the $p$th position is
deleted and at that point $i_p=(r-1)-0+1=r$ so $x^D=x_{r} x^E$.
\end{proof}

\begin{example}\label{ex:T}
If $D$ is the pipe dream on the left in Figure~\ref{fig:bijection},
then the corresponding permutation is $\pi=[1,4,3,2]$. The
lex largest inversion of $\pi$ is $(3,4)$ so $\nu = \pi
t_{3,4} =[1,4,2,3]$.  The $\{3,4\}$-crossing is circled.  Using the
biword encoding of $D$, we have $\mathbf{r}_D=(2,3,2)$ and
$\mathbf{j}_D=(2,2,1)$.  To compute $\TransitionMap(D)$, we initiate
the bump at $t_0=1$ since the $\{3,4\}$-crossing is first in reading order on
$D$.
\[\Bump^{-}_{1}((2,3,2),(2,2,1)) = ((1,3,2),(1,2,1),1,1,\bumped).\]
It requires just one push map since $(1,3,2)$ is reduced.  The
crossing in column 1 in the right-labeled wiring diagram for $(1,3,2)$
is between wires $(1,3)$.  Thus, $\TransitionMap(D)=E \in \rp(\nu
t_{1,3})$ where $E$ is the pipe dream encoded by
$\mathbf{r}_E=(1,3,2)$ and $\mathbf{j}_E=(1,2,1)$.  Observe that $E$
is the second pipe dream in Figure~\ref{fig:bijection}.

\end{example}

\begin{example}
  In Figure~\ref{fig:pipe dream bump}, we give a more complicated
example of computing $\TransitionMap(D)$.  Note that a defect can
occur either above or below the pushed crossing.  Going from the
fourth to the fifth pipe dream, two consecutive pushes on the same row
are combined into one step.  This is an example of a nontrivial
``stack push''.
\end{example}

\begin{figure}
\begin{center}
\includegraphics[width=\textwidth]{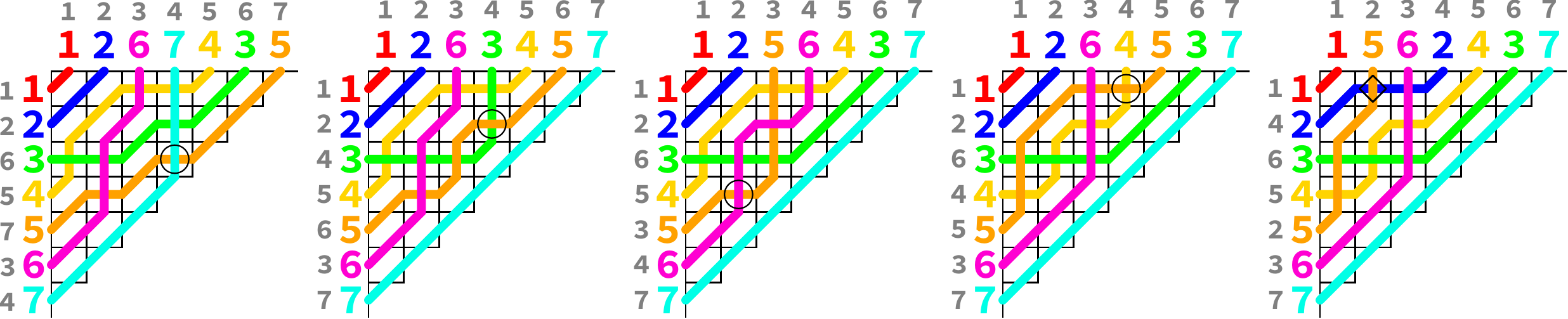}
\caption{If $D$ is the pipe dream on the left, then
  $\TransitionMap(D)$ is the pipe dream on the right. In between we
  show the stack pushes in the bounded bumping algorithm.  The crossing initiating a stack push
  is circled for each step, and the final crossing which
  moved in the last step is marked with a diamond. Here,
  $\pi=[1265734]$, hence $\pi^{-1} = [1267435]$, $r=5, s=7$,
  $\nu = [1265437]$.  In this case, $\TransitionMap(D)$ is a pipe
  dream for $\nu t_{25}=[1465237]$ so $q=2$.
  \label{fig:pipe dream  bump}}
\end{center}
\end{figure}


\begin{lem}~\label{l:deleted.step}
  Assume the notation in Theorem~\ref{t:transitionA} and the
  definition of the transition map $\TransitionMap$. Let $D \in
  \rp(\pi)$.  If
  $$
  \Bump^{-}_{t_0}(\mathbf{r}_{D}, \mathbf{j}_{D}) = (\bfa',\bfb', i,
  j, \deleted),
  $$ then $i=r-1$, $j=\ell(\pi)$,  and $(\mathbf{a}', \bfb')$ encodes a
  pipe dream $E \in\rp(\nu)$ so $\TransitionMap(D) \in \Tset(\pi)$.
\end{lem}

\begin{proof}
The fact that $(\mathbf{a}', \bfb')$ encodes a pipe dream $E
\in\rp(\nu)$ follows directly from
Proposition~\ref{t:little}(\ref{item:3}),
Lemma~\ref{l:stack.push} and the construction of $\nu$ and $t_0$.

To show $i=r-1$ and $j=\ell(\pi)$, it suffices to prove that the last
step of the bounded bumping algorithm stack pushed a crossing in
position $(r,1)$ into the 0th column and out of the pipe dream.  This
is because $r$ is the last descent of $\pi$ so there cannot be any
crossings in the reading order on $D$ after $(r,1)$.

Say the $\{r,s\}$-wire crossing in $D$ is in row $x$ column $y$.  The
bumping algorithm is initiated with a stack push to the left starting
at position $(x,y)$.  The $r$-wire enters $(x,y)$ horizontally since
$r<s$ and $D$ is a reduced pipe dream.  If there is no empty position
in row $x$ to the left of column $y$, then $x=r$ since the wires are
labeled increasing down the left side of $D$ by definition of a pipe
dream.  Otherwise, if $y'$ is the largest column such that $y'<y$ and
$(x,y') \not \in D$, then the initial decrement-stack push on $D$
would result in the pipe dream $E=D-{(x,y)}\cup {(x,y')}$.  Let $s'$
be the wire crossing with the $r$-wire at $(x,y')$ in $E$.  Now, $E$
cannot be reduced since $\outcome=\deleted$ in the bounded bumping
algorithm.  So by Lemma~\ref{lem:nearly_reduced}, we know there exists
exactly one other position in $E$ where wires $r$ and $s'$ cross, say
in position $(x'',y'')$.  By analyzing the possible wire configurations
for two wires to cross exactly twice in a pipe dream, we see that the
$r$ wire is the horizontal wire crossing in the defect position
$(x'',y'')$ in $E$.  Furthermore,
$D -(x,y),\ E -(x,y'),\ E -(x'',y'') \in
\rp(\nu)$.

Update $x \leftarrow x''$, $y \leftarrow y''$ and apply another stack push at
$(x,y)$.  Recursively applying the same argument we see that the only
crossing in $D$ that can be pushed into column $0$ must be in position
$(r,1)$.
\end{proof}


\begin{lem}\label{l:bumped.step}
  Assume the notation in Theorem~\ref{t:transitionA} and the
  definition of the transition map $\TransitionMap$. Let $D \in \rp(\pi)$. While computing $\TransitionMap(D)$,   if
  \[
  \Bump^{-}_{t_0}(\mathbf{r}_{D}, \mathbf{j}_{D}) = (\bfa',\bfb', i,
  j, \bumped),
  \]
  then the crossing in column $j$ of the right-labeled wiring
  diagram of $\mathbf{a}'$ corresponds with the $r$-wire and the
  $q$-wire for some $q<r$ such that $\ell(\pi)=\ell(\nu t_{qr})$.
  Therefore, $(\mathbf{a}', \bfb')$ encodes a pipe dream $E \in\rp(\nu
  t_{qr}) \subset \Tset(\pi)$.
\end{lem}

\begin{proof}
The proof follows from Proposition~\ref{t:little}(\ref{item:2}) and
Lemma~\ref{l:stack.push}.
\end{proof}

Given the notation established in Theorem~\ref{t:transitionA}, we can
state an important consequence of the proof of Little's main theorem
\cite[Theorem 7]{little2003combinatorial}.  We will use this theorem
to prove Lemma~\ref{l:bumped.step.+}.  Recall that Little's bumping algorithm in
the $+1$ direction on
reduced words maps
$\bfa$ maps to $\bfa'$ if and
only if \[ \Bump^{+}_j(\mathbf{a}, \mathbf{b}) =
(\mathbf{a}',\mathbf{b}', i, t, \bumped) \] for any (equivalently every) bounded word
$\bfb$ for $\bfa$.

\begin{thm}\label{thm:little.consequence}\cite{little2003combinatorial}
Assume the notation in Theorem~\ref{t:transitionA}.
Further assume there exists some $q<r$ such that $\ell(\pi)=\ell(\nu t_{qr})$.
Given any bounded pair $(\bfa,\bfb)$ for $\nu t_{qr}$, let $j$ be the
column of the $\{q,r \}$-wire crossing in $\bfa$. If
\[ \Bump^{+}_j(\mathbf{a}, \mathbf{b}) =
(\mathbf{a}',\mathbf{b}', i, t, \bumped), \] then, $\bfa' \in R(\pi)$
and $t$ is the column containing the $\{r,s\}$-crossing in the wiring
diagram of $\bfa'$.
\end{thm}

\begin{lem}\label{l:bumped.step.+}
Assume the notation in Theorem~\ref{t:transitionA}.  Let $E \in
\rp(\nu t_{q,r}) \subset \Tset(\pi)$ for some $1 \leq q <r$.  Assume
$E$ is encoded by the bounded pair
$(\mathbf{r}_{E}, \mathbf{j}_{E})$,
 $j$ is the column
of the $\{q,r \}$-crossing in
$\mathbf{r}_{E}$,
and
\[ \Bump^{+}_j(\mathbf{a},
\mathbf{b}) = (\mathbf{a}',\mathbf{b}', i, t, \bumped).  \] Then
$(\mathbf{a}',\mathbf{b}')$ encodes a pipe dream $D \in \rp(\pi)$
and $t$ is the column containing the $\{r,s\}$-crossing in the wiring
diagram of $\bfa'$.
\end{lem}

\begin{proof}

By Lemma~\ref{l:stack.push}, we know $(\mathbf{a}',\mathbf{b}')$
encodes a reduced pipe dream $D$ for some permutation with the same
length as $\pi$.
%
By Theorem~\ref{thm:little.consequence}, $\bfa' \in R(\pi)$,
so we conclude $D \in \rp(\pi)$.  The column of the $\{r,s
\}$-crossing also follows from the theorem.
\end{proof}


This completes all the lemmas needed for Theorem~\ref{t:transitionA}.
Finally in this section, we introduce the transition chains.  The
transition chains will appear in the examples of the main bijection
for Macdonald's identity.  However, these chains do not appear in the
formal proof of Theorem~\ref{t:macdonald} explicitly.

Recall that to compute a Schubert polynomial $\mathfrak{S}_{\pi }$,
one applies the Transition Equation recursively until the expansion is
given as a positive sum of monomials.  We could also use the
Transition Map $\TransitionMap$ repeatedly to see exactly what happens
to each pipe dream in $\rp(\pi )$ in this process.  The transition
chain records these steps in a way that enables the process to be
reversed.

 \begin{defn}\label{d:young.chain} Given $D \in \rp(\pi)$, define the
associated \textit{transition chain} $Y(D)$ \textit{for} $\pi$
recursively as follows.  If $D$ is the empty pipe dream, then $D \in
\rp(\id)$ and $Y(D):=()$, the empty list.  If $D \in \rp(\pi)$ is not
empty, compute $\TransitionMap(D)=E \in \rp(\pi t_{r,s} t_{q,r})$ and
$Y(E)= ( (q_{k-1},r_{k-1}), \ldots,(i_2,r_2),(i_1,r_1))$.  Prepend
$(q,r)$ onto the front, so that \begin{equation}\label{e:young.chain}
Y(D) : =((q,r),(q_{k-1},r_{k-1}),
\ldots,(q_2,r_2),(q_1,r_1)).
\end{equation}
\end{defn}

For example, starting with the pipe dream $D$ encoded by $((2,3,2),
(1,2,2))$ on the left in Figure~\ref{fig:bijection}, we apply
transition maps along the second row until the we get to the empty pipe
dream.  This sequence of pipe dreams goes along with the following
data.

\[
\begin{array}{|c|c|c|c|c|}
\hline
\pi &	(r,s) &	 (r_{D}, j_{D}) &	(r_{E},j_{E}) &	(q,r)\\
\hline
\left[ 1,4,3,2\right] & (3,4) & (232,122) &	(132,122) &	      (1,3)\\
\left[ 2,4,1,3\right] & (2,4) & (132,122) &	(12,11) & (2,2)\\
\left[ 2,3,1,4\right] & (2,3) & (12,11)  &	(1,1)  &  (2,2)\\
\left[ 2,1,3,4\right] & (1,2) & (1,1) & (,)&	(1,1)  \\
\hline
\end{array}
\]
Thus, $Y(D) = ((1,3),(2,2),(2,2),(1,1))$.

\begin{cor}\label{cor:transition.chain}
 The reduced pipe dreams for $\pi$ and the transition chains
for $\pi$ are in bijection via the map $Y$.
\end{cor}

\begin{proof}
This statement clearly holds for $\id$.  By induction on inversion
order
and Remark~\ref{rem:lex.largest.inversion.order},
we can assume $Y$ maps all pipe dreams in $\Tset(\pi)$
bijectively to their transition chains.  The claim now follows since
$\TransitionMap^{-1}$ is a bijection and the observation that the
computation in Algorithm~\ref{algorithm:invTransitionMap} only relies
the input $E \in \Tset(\pi)$ and the pair $(q,r)$.
\end{proof}

\section{Transition Equation for Bounded Pairs}\label{s:trans.bounded.pairs}

In the last section, the Transition Equation for Schubert polynomials
was proved via a bijection on reduced pipe dreams (which can be
interpreted as bounded pairs of a special kind).  Next we give an
analogue of the Transition Equation for the enumeration of all
bounded pairs of a permutation, together with a bijective proof.  Combining the two
bijections will lead to our bijective proof of
Theorem~\ref{t:macdonald}.  Some terms are defined here in more
generality than we need for this proof -- they will be used later in
Section~\ref{s:specialization}.

\begin{defn}\label{d:bounded.pair.poly} Fix $\pi \in S_{n}$ of length
  $p$.  Given $\mathbf{a}\in R(\pi)$ and $\mathbf{b}$ a bounded word
  for $\mathbf{a}$, let $x^{\mathbf{a}}:=x_1^{a_{1}}x_2^{a_{2}}\cdots
  x_p^{a_{p}}$ and similarly for $y^{\mathbf{b}}$.  Let $\q$ be an
  formal variable, and let $\mathbf{x}=(x_1, x_2, \ldots, x_p)$ and
  $\mathbf{y}=(y_1, y_2, \ldots, y_p)$ be two alphabets of formal
  variables.  Define the \textit{bounded pair polynomial} to be
\[
\Bpoly_{\pi}(\mathbf{x},\mathbf{y};\q) := \sum_{(\bfa,\bfb)}
x^{\mathbf{a}}y^{\mathbf{b}}\q^{\comaj(\mathbf{a})}
\]
where the sum is over all bounded pairs
$(\mathbf{a},\mathbf{b})$ for $\pi$.  For $\pi=\id$, set
$\Bpoly_{\pi}(\mathbf{x},\mathbf{y};\q):=1$.  Thus, setting all of the
variables to 1 gives the number of bounded pairs for $\pi$:
\[
\Bpoly_{\pi}(\mathbf{1})= \Bpoly_{\pi}(\mathbf{1},\mathbf{1};1)=  \sum_{(a_1,a_2,\ldots, a_p) \in R(\pi)} a_1\cdot
a_2 \cdots a_p.
\]
\end{defn}

For example, $\Bpoly_{s_{k}}(\mathbf{1}) =k$ for a simple transposition
$s_{k}$.  Also, $\Bpoly_{[3,2,1]}=6$ as mentioned in the introduction.

\begin{thm}[Transition Equation for Bounded Pairs]\label{t:transitionBP}
For all permutations $\pi$ such that $\ell(\pi )=p>0$, the
number of bounded pairs satisfies the following recursive formula:
\begin{equation}\label{e:transB}
\Bpoly_{\pi}(\mathbf{1})= p\ \Bpoly_{\nu}(\mathbf{1}) \ + \sum_{\substack{
q<r:
\\
\ell(\pi)=\ell(\nu t_{qr})
}}
\Bpoly_{\nu t_{qr}} (\mathbf{1}),
\end{equation}
where $(r,s)$ is the lex largest inversion of $\pi$, and
$\nu=\pi t_{rs}$.  The base cases of the recurrence is $\Bpoly_{\id }(\mathbf{1}) =1$.
\end{thm}

Observe the same permutations appear in the right hand side of
\eqref{e:transB} as in \eqref{e:transA} so
Remark~\ref{rem:lex.largest.inversion.order} applies here as well.

\begin{proof} Similarly to the proof of Theorem~\ref{t:transitionA},
we give a bijective proof of this recurrence by defining a map
$\BoundedPairTransMap$ that maps $\BoundedPairs(\pi)$ to
\[
\Xset(\pi):=
\Bigl( \BoundedPairs(\nu) \times [1,p] \Bigr) \ \cup   \bigcup_{\substack{
q<r:
\\
l(\pi)=l(\nu t_{qr})
}} \BoundedPairs(\nu t_{qr}) \times \{0\}
\]
whenever $p=\ell(\pi)>0$.

\begin{algorithm}[Bounded Transition]
  \label{algorithm:bt}\
Suppose $\pi \neq \id$ is given, and let $(r,s)$ and $\nu$
be defined as in Theorem~\ref{t:transitionBP}.

  \noindent \textbf{Input}: A bounded pair $(\mathbf{a}, \mathbf{b})$
  for $\pi$.

\noindent \textbf{Output}:
$\BoundedPairTransMap(\mathbf{a},\mathbf{b})=((\mathbf{a}',
\bfb'), \dvalue)\in \Xset(\pi)$.

\begin{enumerate}
\item Let $t_{0}$ be the unique column containing the $\{r,s\}$-wire crossing
  in the right-labeled wiring diagram for $\bfa$.

\item Compute $\Bump^{-}_{t_0}(\mathbf{a},\mathbf{b}) =
  (\bfa',\bfb',i, j, \outcome)$.

\item If $\outcome=\deleted$, then $j$ is the last crossing pushed in
the bounded bumping algorithm before being deleted so $j \in [1,p]$.
Return $((\mathbf{a}', \bfb'), j)$ and \textbf{stop}.  By
Proposition~\ref{t:little}(\ref{item:3}), we know $(\mathbf{a}',
\bfb')$ is a bounded pair for $\nu$.

\item If $\outcome=\bumped$, return $((\mathbf{a}', \bfb'), 0)$ and
\textbf{stop}.   Proposition~\ref{t:little}(\ref{item:2}) shows that
one of the wires crossing in column $j$ of the right-labeled wiring
diagram for $\mathbf{a}'$ is the $r$-wire, the other is labeled by
some $q<r$ with $\ell(\pi)=\ell(\nu t_{qr})$.  Therefore,
$(\mathbf{a}', \bfb')$ is a bounded pair for $\nu t_{qr}$.
\end{enumerate}
\end{algorithm}

\begin{algorithm}[Inverse Bounded Transition]
  \label{algorithm:bti}\
Suppose $\pi \neq \id$ is given, and let $(r,s)$ and $\nu$
be defined as in Theorem~\ref{t:transitionBP}.

  \noindent \textbf{Input}: $((\mathbf{e}, \mathbf{f}), \dvalue) \in
  \Xset(\pi)$, in particular $(\mathbf{e}, \mathbf{f})$ is a bounded
  pair for $\nu t_{q,r}$ for some $1 \leq q \leq r$.

\noindent \textbf{Output}: $\BoundedPairTransMap^{-1} ((\mathbf{e},
\mathbf{f}), \dvalue)= (\mathbf{a}, \mathbf{b})$, a bounded pair for $\pi$.

\begin{enumerate}
\item If $q=r$, then $\dvalue\in [1,p]$ and $\mathbf{e}=(e_1,\ldots,
  e_{p-1}) \in R(\nu )$.  If $\dvalue<p$, set $\omega \leftarrow
  s_{e_{p-1}} s_{e_{p-2}} \cdots s_{e_{\dvalue}}$, and if $\dvalue=p$ set
  $\omega \leftarrow \id$.  Set
  \begin{align*}
    i &\leftarrow\omega^{-1}(r), \\
    \mathbf{g} &\leftarrow  \Insert_\dvalue^i(\mathbf{e}),\\
    \mathbf{h} & \leftarrow  \Insert_\dvalue^0(\mathbf{f}),\\
    j & \leftarrow \dvalue.
  \end{align*}

\item If  $q<r$, then $\dvalue=0$.  Set $\mathbf{g} \leftarrow \mathbf{e}$
and $\mathbf{h}\leftarrow \mathbf{f}$.  Let $j$ be the column of the
$\{q,r\}$-wiring crossing in the right-labeled wiring diagram of
$\mathbf{e}$ which must exist since $(\mathbf{e}, \mathbf{f}) \in \BoundedPairs(\nu t_{q,r})
\subset \Xset(\pi)$.

\item Compute $\Bump^{+}_j(\mathbf{g}, \mathbf{h}) =
(\mathbf{g}',\mathbf{h}', i', t, \bumped).$ Return $(\mathbf{g}',
\mathbf{h}')$ and \textbf{stop}.  Note the outcome will always be
$\bumped$ since we are applying
increment-pushes. Lemma~\ref{l:bumped.step.+.gen} below shows that in
all cases, $(\mathbf{g}',\mathbf{h}')$ is a bounded pair for $\pi$.
\end{enumerate}
\end{algorithm}

Observe that $\BoundedPairTransMap$ is an injection since the bounded
bumping algorithm is reversible given the column $j$ of the final push
and the value $r$ in the case $\outcome = \deleted$.  Thus,
$\BoundedPairTransMap^{-1}\BoundedPairTransMap(\bfa,\bfb) =
(\bfa,\bfb)$ so
$\BoundedPairTransMap(\bfa,\bfb)=((\mathbf{e},\mathbf{f}),j)=\BoundedPairTransMap(\bfa',\bfb')$
implies $ (\bfa,\bfb)= (\bfa',\bfb')$.  Also, $\BoundedPairTransMap$
is surjective since $\BoundedPairTransMap^{-1} $ is well defined on
all of $\Xset(\pi)$ and one can observe that $\BoundedPairTransMap
\BoundedPairTransMap^{-1}$ is again the identity map.  Therefore,
$\BoundedPairTransMap: \BoundedPairs(\pi) \longrightarrow \Xset(\pi)$
is a bijection proving the Transition Equation for Bounded Pairs.
\end{proof}

The following lemma is a generalization of Lemma~\ref{l:bumped.step.+}.

\begin{lem}\label{l:bumped.step.+.gen}
  Assume the notation in Theorem~\ref{t:transitionBP}.  Let
  $\bfa=(a_1,\ldots,a_p)$ be any nearly reduced word at $1\leq j\leq
  p$ such that $\Delete_j (\mathbf{a}) \in R(\nu)$ and the wires
 such that $\Delete_j (\mathbf{a}) \in R(\nu)$ and the wires
  crossing in column $j$ of the right-labeled wiring diagram for
  $\bfa$ are labeled by $q<r$ where $r$ is the last descent of $\pi$.
  Let $\bfb$ be a bounded word for $\bfa$.  If
    \[
    \Bump^{+}_j(\mathbf{a}, \mathbf{b}) = (\mathbf{a}',\mathbf{b}', i', t, \bumped),
    \]
    then $(\mathbf{a}',\mathbf{b}')$ is a bounded pair for $\pi$ and
    $t$ is the column containing the $\{r,s\}$-wire crossing in the
    right-labeled wiring diagram of $\bfa'$.
\end{lem}

\begin{proof}
  By assumption, $\Delete_j \mathbf{a} \in
R(\nu)$ so $\Delete_t \mathbf{a}' \in R(\nu)$ by
Proposition~\ref{t:little}(3).  Say wires $\{k,l\}$ with $k<l$ cross
in column $t$ of $\mathbf{a}'$.  Then  $\bfa' \in R(\nu t_{k,l})$.

 By design, the $r$-wire is the larger labeled wire involved in the
crossing in column $j$ of the right-labeled wiring diagram
for $\bfa$, hence the $r$ wire will continue to be one of
the two wires crossed for every increment-push in the
bounded bumping algorithm so $k=r<l$ by
Proposition~\ref{t:little}(2).  Recall,
$\nu(r) < \nu(s)$, \
$\ell(\nu t_{r,s})=\ell(\nu)+1$, and $(r,s)$ is the lex
largest inversion of $\pi=\nu t_{r,s}$, so we
know $\nu(m)<\nu(r)$ for
every $m$ such that $r<m<s$. Thus, $l\geq s$.  Furthermore,
$\nu(s)<\nu(s+1)<\nu(s+2) < \ldots$, so the only possible
value of $l$ such that $\ell(\nu t_{r,l})=\ell(\nu)+1$ is
$l=s$.  Thus, we can conclude $\bfa \in R(\pi)$.
\end{proof}

\begin{defn}\label{d:young.chain.bounded.pair} Given
$(\mathbf{a},\mathbf{b}) \in \BoundedPairs(\pi)$, define the associated
\textit{transition chain} $Y'(\mathbf{a},\mathbf{b})$ recursively as
follows.  If $\pi = \id$, then $Y'(\textbf{a},\textbf{b}):=()$, the
empty list.  If $\pi \neq \id $, compute
$\BoundedPairTransMap(\mathbf{a},\mathbf{b})=
(\mathbf{a}',\mathbf{b}') \in \BoundedPairs (\pi t_{r,s}t_{q,r})
\subset \Xset(\pi)$ and $Y'(\mathbf{a}',\mathbf{b}')= ( (q_{k-1},r_{k-1}),
\ldots,(q_2,r_2),(q_1,r_1))$.  Prepend $(q,r)$ to get
\begin{equation}\label{e:young.chain.bounded.pair} Y'(\mathbf{a},\mathbf{b}) :
=((q,r),(q_{k-1},r_{k-1}), \ldots,(q_2,r_2),(q_1,r_1)).
\end{equation}
\end{defn}

Many bounded pairs for $\pi$ map to the same transition chain via
$Y'$.  For example, when $\pi =[1,4,3,2]$ all 6 of the following
bounded pairs map to $((1, 3), (2, 2), (2,2), (1,1))$ via $Y'$.

\[
\begin{array}{|c|c|}
\hline
\mathbf{a} & \mathbf{b}\\
\hline
3 2 3 &	 1 2 2\\
3 2 3 &	  1 2 3 \\
2 3 2 &	  2 1 1 \\
2 3 2 &	  2 1 2 \\
2 3 2 &	  2 2 1 \\
2 3 2 &	  2 3 1\\
\hline
\end{array}
\]



\section{Bijective proof of Macdonald's formula}\label{s:bij.proof}

In this section, we spell out the promised bijection proving
Macdonald's formula in Theorem~\ref{t:macdonald}.
%
We introduce some notation first and then define the
Macdonald map $M$ on all bounded pairs.

Recall from the
introduction that both sides of the formula can be interpreted
combinatorially. The sum on the left side of \eqref{e:macdonald.formula}
clearly equals $|\BoundedPairs(\pi)|$.  Let
$$\codes(\pi)= [1,1] \times [1,2] \times \cdots \times [1,p]$$ where
$p= \ell(\pi)$ and $[i,j] = \{i,i+1,\ldots,j \}$.  Recall $\mathbf{c}
= (c_1,c_2,\ldots, c_p) \in \codes(\pi)$ is a sub-staircase word of
length $p$.  By Definition~\ref{d:schubs}, one observes that the right
side of \eqref{e:macdonald.formula} equals $|\codes(\pi) \times
 \rp(\pi) |$.  We will refer to the elements $( \mathbf{c}, D) \in
\cup_{\pi \in S_{\infty}} \codes(\pi) \times \rp(\pi)$ as $\cd$-pairs
for $\pi$.

We can now define a map $M$ from all bounded pairs to all $\cd$-pairs
which preserves the underlying permutation.

\begin{algorithm}[Macdonald Map]
  \label{algorithm:mac}\

  \noindent \textbf{Input}:
A bounded pair $(\mathbf{a}, \mathbf{b}) =
((a_1,\ldots, a_p),(b_1,\ldots, b_p))$.  Let $\pi = s_{a_1}s_{a_2}
\cdots s_{a_p}$.  By definition of a bounded pair, $\mathbf{a}$ is
reduced, so $p =\ell(\pi)$.

\noindent \textbf{Output}: $\MacdonaldMap (\mathbf{a},\mathbf{b}) =
( \mathbf{c}, D) \in \codes(\pi) \times \rp(\pi)$.

\begin{enumerate}
\item If $\pi$ is the identity, then we must have $(\mathbf{a}, \mathbf{b}) =
  ((),())$. Set $\mathbf{c}=()$ and $D =\{\}$.   Return $( \mathbf{c}, D)$ and \textbf{stop}.

\item Compute
  $\BoundedPairTransMap(\mathbf{a},\mathbf{b})=((\mathbf{a}', \bfb'),
  \dvalue) \in \Xset(\pi)$. Say $(\mathbf{a}', \bfb')$ is a bounded pair for
  $ \nu t_{q,r}$ where $(r,s)$ is the lex largest inversion
  for $\pi$, $\nu = \pi t_{r,s}$ and $1 \leq q \leq r$.

\item Recursively compute $\MacdonaldMap(\mathbf{a}', \mathbf{b}') =
(\mathbf{c}', D')$.  By induction on inversion order and
Remark~\ref{rem:lex.largest.inversion.order}, we can assume that
$(\mathbf{c}', D') \in \codes(\nu t_{q,r}) \times \rp(\nu t_{q,r})$.

\item Set $D = \TransitionMap^{-1}(D')$.  If $q<r$, set $\mathbf{c} =
\mathbf{c}'$.  Otherwise, if $q=r$, set $\mathbf{c} = \Insert_p^
%
\dvalue
(
\mathbf{c}')$.  Return $( \mathbf{c}, D)$ and \textbf{stop}.  Observe
that in either case, $\mathbf{c} \in \codes(\pi)$. By
Algorithm~\ref{algorithm:invTransitionMap}, $D \in \rp(\pi)$.
\end{enumerate}
\end{algorithm}

\begin{example} Consider the bounded pair
$(\mathbf{a},\mathbf{b}) = ((2, 3, 2), (2, 1, 2 ))$ for the
permutation $\pi =[1,4,3,2]$.  The steps from
Algorithm~\ref{algorithm:mac} in this case are summarized in
Table~\ref{table:232.212}.  The result is
$M((2, 3, 2), (2, 1, 2 ))= (\mathbf{c},D)$ where $\mathbf{c}=(1,1,2)$
and $D$ is the pipe dream encoded by the biword
$(\mathbf{r}_D,\mathbf{j})=((2,3,2), (2,2,1))$.  The transition chain
is $Y(D)=((1,3),(2,2),(2,2),(1,1))$.  Figure~\ref{fig:bijection} from
Section~\ref{s:intro} illustrates the computations in this table using
drawings of pipe dreams and wiring diagrams.
\end{example}

\begin{figure}
\[
\begin{array}{|c|c|c|c|c|}
\hline
\pi & (\mathbf{a},\mathbf{b}) & (q,r) & \dvalue & ( \mathbf{c}, \mathbf{r}_D, \mathbf{j}_D)   \\
\hline
\left[1 4 3 2\right] & (2 3 2, 2 1 2)&  (1, 3) & 0 &  ( 1 1 2, 2 3 2, 2 2 1) \\

\left[2 4 1 3\right] & (1 3 2, 1 1 2)&  (2,2) & 2 &  ( 1 1 2, 1 3 2, 1 2 1) \\

\left[2 3 1 4\right] & (1 2, 1 2)&  (2,2) & 1 &  ( 1 1, 1 2, 1 1) \\

\left[2 1 3 4\right] & (1, 1)&  (1,1) & 1 &  (1, 1, 1) \\

\left[1 2 3 4\right] & (\emptyset, \emptyset) & & & (\emptyset, \emptyset, \emptyset) \\
\hline
\end{array}
\]
\caption{Data for the computation $M((2, 3, 2), (2, 1, 2
  ))$.}\label{table:232.212}
\end{figure}

\begin{example}
  Consider the bounded pair
  $((5, 4, 3, 5, 6, 4, 5), (1, 4, 2, 3, 5, 3, 5))$ for
  $\pi =[1, 2, 6, 5, 7, 3, 4]$.
The steps from
Algorithm~\ref{algorithm:mac} in this case are summarized in
  Table~\ref{table:5435645}.  The result is the pair
  $( \mathbf{c}, D)=(\mathbf{c}, \mathbf{r}_D, \mathbf{j}_D)$ where
  $D$ is the pipe dream $D$ on the far left in
  Figure~\ref{fig:pipe dream bump} and $\mathbf{c}=(1,1,1,3,2,1,3)$.
  Again, in the table, we use the biword encoding of pipe dreams,
  $(\mathbf{r}_D, \mathbf{j}_D)$.
\end{example}

\begin{figure}
\[
\begin{array}{|c|c|c|c|c|}
\hline
\pi & (\mathbf{a},\mathbf{b}) & (q,r) & \dvalue & ( \mathbf{c}, \mathbf{r}_D, \mathbf{j}_D)   \\
\hline
\left[1 2 6 5 7 3 4\right] & (5 4 3 5 6 4 5, 1 4 2 3 5 3 5)&  (2, 5) & 0 &
( 1 1 1 3 2 1 3 , 4 3 5 6 4 3 5, 4 3 4 4 2 1 2) \\

\left[1 4 6 5 2 3\right] & (5 3 2 4 5 3 4, 1 3 1 2 4 2 4)&  (4,4) & 3 &  (  1 1 1 3 2 1 3, 3 2 4 5 3 4,  2 3 3 2 1 1) \\

\left[1 4 6 3 2 5\right] & (5 2 3 4 2 3, 1 2 1 3 1 3)&  (1, 4) & 0 & ( 1 1 1 3 2 1, 3 2 3 5 4 3, 3 2 2 3 2 1) \\

\left[ 4 6 1 3 5\right] & (5 1 3 4 2 3, 1 1 1 3 1 3)&  (3,3) & 1 & ( 1 1 1 3 2 1, 3 1 2 5 4 3, 3 1 1 3 2 1) \\

\left[2 4 5 1 3\right] & (1 3 4 2 3, 1 1 3 1 3)&  (3,3) & 2 &  (1 1 1 3 2, 3 1 2 4 3, 3 1 1 2 1) \\

\left[2 4 3 1\right] & (1 3 2 3, 1 2 1 3)&  (3,3) & 3 &  ( 1 1 1 3, 3 1 2 3, 3 1 1 1) \\

\left[2 4 1 3\right] & (1 3 2, 1 2 2)&  (1, 2) & 0 &  ( 1 1 1, 3 1 2, 3 1 1) \\

\left[3 2 1 4\right] & (1 2 1, 1 1 1)&  (2,2) & 1 &  ( 1 1 1, 2 1 2, 2 1 1) \\

\left[3 1 2 4\right] & (2 1, 1 1)&  (1,1) & 1 &  ( 1 1, 2 1, 2 1) \\

\left[2 1 3 4\right] & (1, 1)&  (1,1) & 1 &  (1, 1, 1) \\

\left[1 2 3 4\right] & (\emptyset, \emptyset) & & & (\emptyset, \emptyset,\emptyset) \\
\hline
\end{array}
\]
\caption{Data for the computation $M((5, 4, 3, 5, 6, 4, 5), (1, 4, 2, 3, 5, 3, 5))$. }\label{table:5435645}
\end{figure}

\begin{proof}[Proof of Theorem~\ref{t:macdonald}]
  Define $\MacdonaldMap_{\pi}$ to be the restriction of
  $\MacdonaldMap$ to $\BoundedPairs(\pi)$.  We will show by induction
  that $\MacdonaldMap_{\pi}$ is a bijection from $\BoundedPairs(\pi)$
  to $\codes(\pi) \times \rp(\pi)$, as required to prove the theorem.

  For the base case, if $\pi = \id$, then $\MacdonaldMap_{\id}$ is the
  bijection mapping $((),()) \mapsto ((),\{\})$. Now assume that
  $\MacdonaldMap_{\omega}: \BoundedPairs(\omega) \longrightarrow
  \codes(\omega) \times \rp(\omega)$ is a bijection for all
  permutations $\omega$ such that $\omega \lessthan \pi$, as in
  Remark~\ref{rem:lex.largest.inversion.order}.

By Theorem~\ref{t:transitionBP}, we know that $\BoundedPairTransMap: \BoundedPairs(\pi)
\longrightarrow \Xset(\pi)$ is a bijection where
  \[
  \Xset(\pi)= \Bigl( \BoundedPairs(\nu) \times [1,p] \Bigr) \ \cup
  \bigcup_{\substack{ q<r: \\ l(\pi)=l(\nu t_{qr}) }}
  \BoundedPairs(\nu t_{qr}) \times \{0\}.
\]
Let
  \[
\Yset(\pi)=
\Bigl(  \codes(\nu)  \times \rp(\nu) \times [1,p] \Bigr) \
\cup
\bigcup_{\substack{
q<r:
\\
l(\pi)=l(\nu t_{qr})
}}
 \codes(\nu t_{qr})   \times  \rp(\nu t_{qr}) \times \{0\}.
\]
The induction hypothesis implies that the restricted map
$\MacdonaldMap \times \id: \Xset(\pi) \longrightarrow \Yset(\pi)$ is a
bijection preserving the underlying permutation.
That is, if $(\bfa,\bfb,\dvalue) \in \Xset(\pi) $, then $\bfa \in R(\nu t_{q,r})$ for some $q \leq r$, and if $\MacdonaldMap \times
\id(\bfa,\bfb,\dvalue)  = (\mathbf{c},D,\dvalue)$ then $D \in \rp(\nu t_{q,r})$
as well.

Define $\RT: \codes(\pi) \times \rp(\pi) \longrightarrow \Yset(\pi)$
by mapping
$$\RT (\mathbf{c}, D)=
\begin{cases}
(\widehat{\mathbf{c}}, \TransitionMap(D),c_p) & |D|>|\TransitionMap(D)|
  \\ \\ ( \mathbf{c}, \TransitionMap(D),0) & |D|=|\TransitionMap(D)|,
  \end{cases}
$$
where $\widehat{\mathbf{c}}:=(c_1,\ldots,c_{p-1})$.
\noindent Since $\TransitionMap$ is a bijection,
$\RT$ is a bijection with well defined inverse $\RT^{-1}:
\Yset(\pi) \longrightarrow \codes(\pi) \times \rp(\pi)$.

The map $\MacdonaldMap_\pi: \BoundedPairs(\pi) \longrightarrow
\codes(\pi) \times \rp(\pi)$ can be written as the composition of
three bijections, $\MacdonaldMap_\pi = \RT^{-1} \circ (\MacdonaldMap
\times \id) \circ \BoundedPairTransMap$, hence is itself a bijection.
This concludes the induction.
\end{proof}

Observe from the proof above, we can show by induction that
$M_{\pi}^{-1} = \BoundedPairTransMap
^{-1}
 \circ (\MacdonaldMap^{-1} \times
\id) \circ \RT$.  Thus, we can write out the algorithm for $M^{-1}$
analogously with Algorithm~\ref{algorithm:mac}.

\begin{cor}
The inverse of $M$ is given by
Algorithm~\ref{algorithm:inverse.mac}.
\end{cor}

\begin{algorithm}[Inverse Macdonald Map]
   \label{algorithm:inverse.mac}\

   \noindent \textbf{Input}: $(\mathbf{c}, D)$ where $D$ is a reduced
   pipe dream for some permutation $\pi$ and
   $\mathbf{c}=(c_1,c_2,\ldots, c_p)$ is a sub-staircase word  of length
   $p =\ell(\pi)$.

 \noindent \textbf{Output}: $\MacdonaldMap^{-1}(\mathbf{c},D)=
 (\mathbf{a},\mathbf{b})$, a bounded pair for $\pi$.

 \begin{enumerate}
 \item If $\pi$ is the identity, then we must have $\mathbf{c}=()$ and
   $D =\{\}$.  Return $(\mathbf{a}, \mathbf{b}) =((),())$ and
   \textbf{stop}.

\item Compute $\TransitionMap(D)=D' \in \Tset(\pi)$.  Say $D'$ is a
  reduced pipe dream for $\nu t_{q,r}$ where $(r,s)$ is the
  lex largest inversion for $\pi$, $\nu = \pi t_{r,s}$ and $1 \leq q
  \leq r$.  If $q<r$, set $\dvalue =0$ and
  $\mathbf{c}'=\mathbf{c}$. Otherwise, if $q=r$, set
$\dvalue = c_p$ and
  $\mathbf{c}'=(c_1,\ldots, c_{p-1})$.

 \item Recursively compute $\MacdonaldMap^{-1} (\mathbf{c}', D') =
(\mathbf{a}', \mathbf{b}')$.  By induction on inversion order and
Remark~\ref{rem:lex.largest.inversion.order}, we can
assume that $(\mathbf{a}', \mathbf{b}') \in \BoundedPairs(\nu t_{q,r})$.

 \item Compute $\BoundedPairTransMap^{-1}((\mathbf{a}', \mathbf{b}'),
\dvalue) =(\mathbf{a}, \mathbf{b})$.  Return $(\mathbf{a}, \mathbf{b})$ and
\textbf{stop}.  By the proof of Theorem~\ref{t:transitionBP}, we know
that $(\mathbf{a}, \mathbf{b})$ is a bounded pair for $\pi$.

 \end{enumerate}
 \end{algorithm}

\begin{remark}
  Observe that in step 2 of both Algorithm~\ref{algorithm:mac} and
  Algorithm~\ref{algorithm:inverse.mac}, the data $(q,r)$ and $\dvalue$ are
  determined from the input.  So if $\MacdonaldMap(\mathbf{a},
  \mathbf{b}) = ( \mathbf{c}, D)$, then by step 4 of both algorithms,
  these 3 quantities, $\dvalue,q,r$ are the same.  In particular, $Y(D)=Y'(\mathbf{a},
  \mathbf{b})$.
\end{remark}

\begin{remark}
In general, it is not easy to ``eyeball'' the map $M$ or $M^{-1}$ by
simply straightening out or bending wires without passing through the
transition chain.  Every biword coming from a reduced pipe dream for
$\pi$ is a bounded pair for $\pi$, but the converse does not hold.
Thus, the map $M^{-1}$ rarely acts as the identity map.  In fact, we
know $M$ is a $p!$ to 1 map if we project the image onto the pipe
dreams.
\end{remark}

\section{$\q$-analogue of Macdonald's formula}\label{s:specialization}

In this section, we prove Theorem~\ref{t:macdonald.q.analog}
bijectively.  The first step is to rewrite the left side
of~\eqref{e:fomin.stanley.formula} as a specialization of the bounded
pair polynomial $\Bpoly_\pi(\mathbf{x},\mathbf{y};\q)$ defined in
Definition~\ref{d:bounded.pair.poly}.  We then prove that this
specialized polynomial satisfies a $\q$-analog of the Transition
Equation for Bounded Pairs, and thus argue that every step of our
algorithmic bijection $\BoundedPairTransMap$ respects the $\q$-weight
so $\MacdonaldMap$ is a $\q$-weight preserving bijection in addition to
preserving the underlying permutation.

Specializing each $x_i=\q$ and $y_i=\q^{-1}$ in
$\Bpoly_\pi(\mathbf{x},\mathbf{y};\q)$, where the third parameter is the
same formal variable $\q$, gives a one parameter version of the bounded pair
polynomial

\[
 \SpecializedBpoly_{\pi}(\q):
 = \sum \q^{(\mathbf{a},\mathbf{b})} =  \sum_{(a_1,a_2,\ldots, a_p) \in R(\pi)} \q^{\comaj(\mathbf{a})}[a_1]\cdot
 [a_2] \cdots [a_p]
\]
 where the first sum is over all bounded pairs
 $(\mathbf{a},\mathbf{b})$ for $\pi$ and $\q^{(\mathbf{a},\mathbf{b})}$ is  the \textit{combined weight}
\begin{equation}\label{e:combined.weight}
\q^{(\mathbf{a},\mathbf{b})}:=\q^{\comaj(\mathbf{a})}\prod_{i=1}^p
 \q^{a_i -b_i}.
\end{equation}

 For example, let $s_r$ be a simple transposition for some $r\geq 1$.
 Then $\SpecializedBpoly_{s_{r}}(\q)=[r]=1+ \q + \q^2 +\ldots + \q^{r-1}$.
 See also the example after Theorem~\ref{t:macdonald.q.analog}.

\begin{thm}[$\q$-Transition Equation for Bounded Pairs]\label{t:transitionBP.qanalog}
  For all permutations $\pi$ such that $\ell(\pi )=p>0$, the
  polynomials $\SpecializedBpoly_{\pi}(\q)$ satisfy the following recursive formula:
\begin{equation}\label{e:trans.qanalog}
\SpecializedBpoly_{\pi}(\q)= (1+\q+\cdots + \q^{p-1}) \q^{r-1} \SpecializedBpoly_{\nu}(\q) \ + \sum_{\substack{
q<r
\\
l(\pi)=l(\nu t_{qr})
}}
\SpecializedBpoly_{\nu t_{qr}} (\q)
\end{equation}
where $(r,s)$ is the lex largest inversion in
$\pi$, and $\nu=\pi t_{rs}$.  The base case of the
recurrence is $\SpecializedBpoly_{\id}(\q)=1$.
\end{thm}

A bijective proof of Theorem~\ref{t:transitionBP.qanalog} implies a
bijective proof of Theorem~\ref{t:macdonald.q.analog} since $\SpecializedBpoly_{\pi}(\q)$
is by definition the left side of \eqref{e:fomin.stanley.formula},
while the right side satisfies the same recurrence and base
conditions as \eqref{e:trans.qanalog} by the Transition Equation for
Schubert polynomials, Theorem~\ref{t:transitionA}.

To prove Theorem~\ref{t:transitionBP.qanalog}, we need to understand
how the maps $\BoundedPairTransMap$ and $\BoundedPairTransMap^{-1}$
change the combined weight for bounded pairs.  This
involves an investigation of the $\comaj$ statistic on reduced words.
Recall for motivation the well known formula due to MacMahon relating
the number of inversions to the $\comaj$  statistic
\begin{equation}\label{eq:macmahon}
\left(1+ \q + \cdots + \q^{n-1}\right) \sum_{\nu \in S_{n-1}}
\q^{\ell(\nu)} = \sum_{\pi \in S_n} \q^{\ell(\pi)} = \sum_{\pi \in S_n}
\q^{\comaj(\pi)}.
\end{equation}
The first equality follows simply by inserting $n$ into the one-line
notation for $\nu\in S_{n-1}$ and observing the change in the number
of inversions.  The second equality can similarly be proved using the
code of a permutation and the Carlitz bijection \cite{carlitz.1975},
see also \cite{gilespe.2016,skandera.2001}.  Note, the Carlitz
bijection is different than Foata's famous bijective proof of the
second equality \cite{Foata}.

We next state a mild generalization of a lemma due to Gupta
\cite{gupta.1978} about how $\comaj$ changes when one additional
letter is inserted into a word in every possible way.  Gupta's proof
covers the case where the numbers $a_k$ are all distinct.
Lemma~\ref{l:comaj} below extends this to sequences with no two
adjacent values equal.  We include a short proof of Gupta's lemma
below as a prelude to extending this analysis to reduced words in
Lemma~\ref{l:aug.comaj.diff} and the proof of
Theorem~\ref{t:transitionBP.qanalog}.  For another proof and further
applications of ``insertion lemmas'' in the liturature, see
\cite{Haglund-Loehr-Remmel.2005,Novick.2010}.  

Fix any sequence of real numbers $\mathbf{a}=(a_{1},\ldots
,a_{p})$. For $1\leq i\leq p+1$, let
$$\mathbf{a}^{j}_{i}:=\Insert_i^j(\mathbf{a}) =
(a_{1},\ldots ,a_{i-1}, j, a_{i},\ldots , a_{p})$$ be the
result of inserting $j$ into $\mathbf{a}$ to become column
$i$.  We also extend the definition of the comajor index to
arbitrary real sequences:
$\comaj(\mathbf{a}):=\sum_{i:a_i<a_{i+1}}i$.


\begin{lem}\cite{gupta.1978}\label{l:comaj}
Fix any sequence of real numbers $\mathbf{a}=(a_{1},\ldots ,a_{p})$
such that no two adjacent elements $a_i, a_{i+1}$ are equal, and let $j$ be a real number
different from $\{a_{1},\dots , a_{p} \}$.  Then, we have
\[
\bigl\{\comaj(\mathbf{a}^{j}_i) - \comaj(\mathbf{a}) \given 1\leq i\leq p+1 \bigr\} =
\{0,1,2,\ldots, p \}.
\]
\end{lem}

For example, take $\mathbf{a}=(2,3,5,2)$ and $j=4$.  Then
$\comaj(\mathbf{a})=3$.  The five words obtained from
$\mathbf{a}$ by inserting 4 in all columns are given below
with their comaj.
\[
\begin{array}{|c|c|c|c|}
\hline
i &  \bfa^{4}_i & \comaj(\bfa^{4}_i) & \comaj(\bfa^{4}_i) - \comaj (\bfa)\\
  \hline
1 & 42352 & 5 & 2\\
2 & 24352 & 4 & 1\\
3 & 23452 & 6 & 3\\
4 & 23542 & 3 & 0\\
5 & 23524 & 7 & 4\\
\hline
\end{array}
\]
As one can see, the difference in comaj takes on the five values from
0 to 4 in permuted order.  We refer to the word $(2,1,3,0,4)$ as the
\emph{comaj difference word}.

The conclusion of the lemma does not hold when adjacent
equal elements are allowed even if we extend the
definition of $\comaj$ to cover weak ascents.  For example,
if $\mathbf{a}=(1,1)$ and $j=1$, then
$\comaj(\mathbf{a}^{j}_i)-\comaj(\mathbf{a})$ is constant
for all $i=0,1,2$.

\begin{proof}
  The statement is easily checked for $p=0,1$.  Let $k$ be a real
  number distinct from $\{j,a_p\}$, and let
  $\mathbf{a'} =(a_{1},\ldots ,a_{p},k)$.  Assume by induction on
  $p \geq 1$ that the statement holds for $\mathbf{a}$, so we can assume
  $\{\comaj(\mathbf{a}^{j}_i)-\comaj(\mathbf{a}) \given 1\leq i\leq p+1
  \} = \{0,1,2,\ldots, p \}.$

  The final element in the comaj difference word for $\mathbf{a}$ can
  be determined from the relative order of $a_p,j$.  If $a_{p}<j$,
  then $\comaj(\mathbf{a}_{p+1}^j)-\comaj(\mathbf{a})=p$ and if
  $a_{p}>j$, then $\comaj(\mathbf{a}_{p+1}^j)-\comaj(\mathbf{a})=0$.

Next, consider the relative order of $a_{p}, j, k $ and how it affects
the comaj difference word.  The possible orders correspond with the 6 permutations
in $S_3$.  For instance, if $j<a_{p}<k$, then
$\comaj(\mathbf{a'}^{j}_i)-\comaj(\mathbf{a'})=\comaj(\mathbf{a}^{j}_i)-\comaj(\mathbf{a})+1$
for all $1\leq i\leq p$ since $k$ adds one new ascent to the right of all
these columns which gets shifted over when $j$ is inserted.  Since
$j<a_{p}$, $\comaj(\mathbf{a}_{p+1}^j)-\comaj(\mathbf{a})=0$ as noted
above so $\{\comaj(\mathbf{a'}^{j}_i) - \comaj(\mathbf{a'}) \given 1\leq
i\leq p \} = \{2,\ldots, p+1 \}$
by the induction hypothesis.
Furthermore, since $j<a_{p}<k$,
$\comaj(\mathbf{a'}_{p+1}^j)-\comaj(\mathbf{a'})=1$ and
$\comaj(\mathbf{a'}_{p+2}^j)-\comaj(\mathbf{a'})=0$.  Thus, the claim
holds for $\mathbf{a'}$ in this case as well.

Each of the remaining 5 cases is similar.  They only depend on the
relative order of $a_{p}, k, j$ and not on any of the specific
values $a_{1},\dots , a_{p}, j, k$.
We leave the remaining cases to the reader or their computer to check.
\end{proof}

We can now use Lemma~\ref{l:comaj} to give a bijective
proof of MacMahon's formula, Equation~\eqref{eq:macmahon}.
A similar argument is implicit in~\cite{carlitz.1975}.

\begin{cor} For all $n\geq 2$,
$$
\sum_{\pi \in S_n} \q^{\comaj(\pi)} =
(1+\q+\cdots + \q^{n-1}) \sum_{\nu \in S_{n-1}} \q^{\comaj(\nu)}.
$$
\end{cor}

\begin{proof}
  For each permutation $\nu \in S_{n-1}$ written in one-line notation,
  there are $n$ ways to insert $n$.  By Lemma~\ref{l:comaj}, the comaj
  statistic will increase by a distinct value in $\{0,1,\ldots, n-1\}$
  for each of these ways.
\end{proof}

Next we prove a variation of Lemma~\ref{l:comaj} involving reduced words.
The idea is similar to the proof of Lemma~\ref{l:comaj}, though the proof is much more technical.

Given a reduced word $\mathbf{a}=(a_{1},\ldots ,a_{p})$,
draw its \textbf{left}-labeled wiring diagram, as in the
first diagram of Figure~\ref{fig:insert.at.2}. Fix a
positive integer $j$ and consider the $j$-wire.  Let
$h^j_i(\mathbf{a})$ be the row of the $j$-wire in column
$i-1$, so $h^j_i(\mathbf{a})=s_{a_{i-1}}\cdots
s_{a_2}s_{a_1}(j)$. In the notation of
Section~\ref{ss:reduced words},
$h^j_i(\bfa)=\pi_{i-1}^{-1}(j)$, where $\pi_t$ denotes the
permutation at time $t$ of $\bfa$. We insert a new crossing
in column $i$ with its left foot meeting the $j$-wire;
i.e. \  define
$$\widetilde{\mathbf{a}}:=\Insert_{i}^{h^j_i(\mathbf{a})-1}(\bfa) =
(a_1,a_2,\ldots,a_{i-1},h^j_i(\mathbf{a})-1,a_i,\ldots,
a_p).$$  Then, $\widetilde{\mathbf{a}}$ may or may not be a reduced
word itself, but it is nearly reduced at $i$.  Now we want
to apply a Little bump to $\widetilde{\mathbf{a}}$.  To be consistent
with our earlier definitions, we write this in terms of
\emph{bounded} bumping: every word is a bounded word for
itself, so we can apply the bounded bumping algorithm with
input $(\widetilde{\mathbf{a}}, \widetilde{\mathbf{a}}, i, +)$. Say
$\Bump_{i}^{+}(\widetilde{\mathbf{a}}, \widetilde{\mathbf{a}}) =
(\widetilde{\mathbf{a}}',\widetilde{\mathbf{b}}',g,h,\outcome)$. Set
\[
y_i^j(\mathbf{a}):=\widetilde{\bfa}'.
\]

\begin{figure}
\begin{center}
  \includegraphics[width=\textwidth]{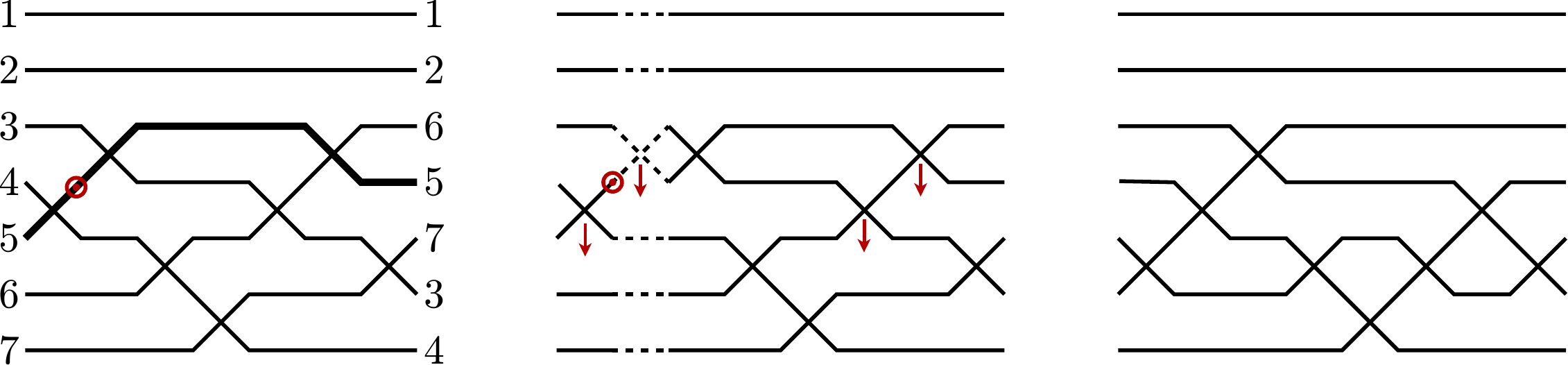} \caption{An
example of the transformation from $\mathbf{a}$ (left) to
$\widetilde{\mathbf{a}}$ (middle) to $\widetilde{\mathbf{a}}'$
(right), with $i=2$ and $j=5$.  The row of wire $5$ just before the
$2$nd crossing is $h_2^5(\mathbf{a})=4$, so we insert a crossing on
row 3 in such a way as to become column 2.  Then we apply an
increment-bump to this crossing to obtain $\widetilde{\mathbf{a}}'$.}
\label{fig:insert.at.2}
\end{center}
\end{figure}

Define the \textit{augmented comaj difference word} for
$\mathbf{a}$ along wire $j$ to be
$\mathbf{v}^j(\mathbf{a}):=(v^j_1(\mathbf{a}), \ldots,
v^j_{p+1}(\mathbf{a}))$ where
$$v^j_i(\mathbf{a}) = \comaj(y_i^j(\mathbf{a})) - \comaj(\mathbf{a})
+h^j_i(\mathbf{a}) -1.$$ The augmented comaj difference
word measures the change in the power of $\q$ in the
combined weight when the $\BoundedPairTransMap$  algorithm
ends in a deletion in each column of step 2(a).

Consider the running example of the reduced word
$\mathbf{a}=(4,3,5,6,4,3,5)$, and fix $j=5$. We compute
$\comaj(\mathbf{a}) = 11$.  The wiring diagram for $\bfa$ is shown in
the first diagram of Figure~\ref{fig:insert.at.2}.  Observe that the
row of the $5$-wire in the wiring diagram decreases to 3 and then
increases to 4 in matrix coordinates.  In the second diagram in
Figure~\ref{fig:insert.at.2}, we show the wiring diagram for
$\widetilde{a}=(4,3,3,5,6,4,3,5)$ computed by inserting an extra
crossing in the second column with its left foot on the $5$-wire, so
inserting a 3 into $\bfa$.  The arrows indicate pushes in the bounded
bumping algorithm for $\Bump_{2}^{+}(\widetilde{\mathbf{a}},
\widetilde{\mathbf{a}})$; they occur in columns 2,1,7,6 in sequential
order.  The third diagram shows the wiring diagram of $y_{2}^{5}(\bfa) = \widetilde{\mathbf{a}}'
= (5,4,3,5,6,5,4,5)$.  Compute $\comaj(y_{2}^{5}(\bfa)) = 14$, so
$v_{2}^{5}(\bfa) =14-11+3 =6$.
Next we display the data to compute
the augmented comaj difference word for $\bfa$.

$$
\begin{array}{|c|c|c|c|c|c|c|}
\hline
i & h^5_i(\mathbf{a}) &  \text{insert}  & y_i^5(\mathbf{a}) & \comaj(y_i^5(\mathbf{a}))  &h^5_i(\mathbf{a})-1  & v^5_i(\mathbf{a}) \\
\hline
1 & 5 & 4 4 3 5 6 4 3 5 & 5 4 3 5 6 5 4 5 & 14 &  4 & 7 \\
2 & 4 & 4 3 3 5 6 4 3 5 & 5 4 3 5 6 5 4 5 & 14 &  3 & 6 \\
3 & 3 & 4 3 2 5 6 4 3 5 & 5 4 3 5 6 5 4 5 & 14 &  2 & 5 \\
4 & 3 & 4 3 5 2 6 4 3 5 & 5 4 5 3 6 5 4 5 & 13 &  2 & 4 \\
5 & 3 & 4 3 5 6 2 4 3 5 & 5 4 5 6 3 5 4 5 & 17 &  2 & 8 \\
6 & 3 & 4 3 5 6 4 2 3 5 & 5 4 5 6 5 3 4 5 & 18 &  2 & 9 \\
7 & 4 & 4 3 5 6 4 3 3 5 & 5 4 5 6 5 3 4 5 & 18 &  3 & 10 \\
8 & 4 & 4 3 5 6 4 3 5 3 & 4 3 5 6 4 3 5 4 & 11 &  3 & 3 \\
\hline
\end{array}
$$
Thus, $v^5(\mathbf{a})=(7, 6, 5, 4, 8, 9, 10, 3)$.  Further examples
appear in the Appendix which demonstrates the computation of
$\mathbf{v}^5
(\mathbf{a})$ for all the initial substrings of $(4, 3,
5, 6, 4, 3, 5)$.

\begin{lem}\label{l:aug.comaj.diff}
  Given a reduced word $\mathbf{a}=(a_{1},\ldots ,a_{p})$ and a fixed
  positive integer $j$, the augmented comaj difference word
  $\mathbf{v}^j(\mathbf{a})=(v^j_1(\mathbf{a}), \ldots, v^j_{p+1}(\mathbf{a}))$
  is a permutation of the integers in the closed interval $[h^j_{p+1}(\mathbf{a}) -1,
  h^j_{p+1}(\mathbf{a}) +p -1]$.  Moreover, every entry of
  $\mathbf{v}^j(\mathbf{a})$ is a record,
  i.e.\ it is either greater than all preceding entries or
  less than all preceding entries.
\end{lem}

The proof is by induction based on initial substrings of $\mathbf{a}$,
similar to the proof of Lemma~\ref{l:comaj}.  It is complicated by the
fact that we are inserting a crossing along the $j$-wire so the value
which we are inserting varies with column.

\begin{proof}


We  prove the statement by induction on $p$.  The statement is
true for the empty reduced word since $\mathbf{v}^j() = (j-1)$.  There
are 4 cases to check if $p=1$.  Say $\mathbf{a}=(a_{1})$, so
$h^j_1(\mathbf{a})=j$ and $h^j_2(\mathbf{a})=s_k(j).$ Then,
\[
\mathbf{v}^j((a_{1}))= \begin{cases}
(j,j-1)  & j<a_{1}  \\
(j,j+1)  & j=a_{1}\\
(j-1,j-2)  & j=a_{1}+1\\
(j-1,j)  & j>a_{1}+1.
\end{cases}
\]
Thus, all 4 cases for $p=1$ satisfy the statements in
the lemma.

Assume the lemma holds by induction for all reduced words up to length
$p\geq 1$.  We will show it holds for all reduced words of length $p+1$.
Let
\[
\mathbf{a'}=(a_{1},\ldots ,a_{p},k)
\]
be a reduced word extending $\mathbf{a}$.  Thus, $h^j_i(\mathbf{a'})=
h^j_i(\mathbf{a})$ for $1\leq i \leq p+1$.  Let
\[
h=h^j_{p+1}(\mathbf{a})=s_{a_{p}}\cdots s_{a_{1}}(j),
\]
then $h^j_{p+2}(\mathbf{a'}) = s_k(h)$.

Our goal is to compute the augmented comaj differences
$v^j_i(\mathbf{a}')$ for $1\leq i\leq p+2$.  We will treat the three
cases $1\leq i\leq p$, $i=p+1$ and $i=p+2$ separately.

First consider the case $1 \leq i \leq p$.  The bounded bumping
algorithm preserves the ascent set of a word by
Proposition~\ref{t:little}(4), so one can observe
that $$\comaj(y_i^j(\mathbf{a})) =
\comaj(a_1,a_2,\ldots,a_{i-1},h^j_i(\mathbf{a})-1/2,a_i,\ldots,
a_p).$$ This fact will allow us to compute the augmented comaj differences
without knowing the exact sequence of pushes required in the bounded
bumping algorithm.
Using the above
observation,
\begin{align*}
\comaj(y_i^j(\mathbf{a'})) =&\comaj(y_i^j(\mathbf{a}))+ (p+1)\cdot \delta,
\end{align*}
while
$$\comaj(\mathbf{a'}) = \comaj(\mathbf{a}) + p\cdot \delta,
$$
 where
$$\delta:=\begin{cases}
1,&a_p<k,\\
0,&\text{otherwise.}
\end{cases}
$$
Since also
$h^j_i(\mathbf{a'})= h^j_i(\mathbf{a})$, we conclude by combining the
last three equations that
$$v_i^j(\mathbf{a'}) = v_i^j(\mathbf{a}) + \delta.$$
We compute $v_{p+1}^j(\mathbf{a})=\comaj(y_{p+1}^j(\mathbf{a})) -
\comaj(\mathbf{a}) +h^j_{p+1}(\mathbf{a}) -1 = p\cdot \delta + h -1.$
By induction, we know $(v_1^j(\mathbf{a}), \ldots, v_p^j(\mathbf{a}))$
is a permutation with every entry being a record of the interval
$[h-1,h +p -1 ] \setminus \{p\cdot \delta + h -1\}$.  Therefore
$(v_1^j(\mathbf{a'}), \ldots, v_p^j(\mathbf{a'}))$ is in fact a
permutation of an interval of consecutive integers such that every
entry is a record.

Now we consider the case $i=p+1$.  We claim that the value
$v_{p+1}^j(\mathbf{a'})$ is completely determined by the
values $ h,a_p,k,p$ as follows.  Note that $a_p$ exists since $p\geq
1$ by assumption, and $a_p \neq k$
since they are adjacent in a reduced word. All possible
cases are
\[
v_{p+1}^j(\mathbf{a'}) = \begin{cases}
h+p & a_p>k\geq h\\
h-1 & a_p<k<h\\
p\cdot \delta + h -1 +\delta  & \text{otherwise. }
\end{cases}
\]
We conclude here that $(v_1^j(\mathbf{a'}), \ldots,
v_{p+1}^j(\mathbf{a'}))$ is in fact a permutation of an interval of
consecutive integers such that every entry is a record.  In the case
$a_p>k\geq h$, the interval is $[h,h+p]$, and in the other two cases
the interval is $[h-1,h+p-1]$.

Finally, consider the case $i=p+2$.  Again, the value
$v_{p+2}^j(\mathbf{a'})$ is straightforward to calculate from the
definition of the augmented comaj vector given $k$ and the fact that
$s_{k}(h)=h^j_{p+2}(\mathbf{a'})$ mentioned above:
$$
v_{p+2}^j(\mathbf{a'}) = \begin{cases}
s_{k}(h)-1 & k\geq s_k(h)\\
p+s_{k}(h) & k < s_k(h).
\end{cases}
$$
Thus, $v_{p+2}^j(\mathbf{a'})$ will be an extreme value in the
interval $[s_{k}(h) -1, s_{k}(h) +p]$ as required for the lemma.  All
that remains to prove the lemma is to ascertain how
$v_{p+2}^j(\mathbf{a'})$ relates to $[h,h+p]$ when $a_{p}>k\geq h$ or
$[h-1,h+p-1]$ otherwise.  This again breaks into cases depending on if
$s_{k}(h)=h,h-1,h+1$.  We leave this straightforward verification to
the reader.
\end{proof}

\bigskip

\begin{proof}[Proof of Theorem~\ref{t:transitionBP.qanalog}] As
mentioned in the introduction to this section, we will show that the
bijection $\BoundedPairTransMap$ from Algorithm~\ref{algorithm:bt}
preserves the $\q$-weight in the following sense.  Assume
$\BoundedPairTransMap(\mathbf{a},\mathbf{b}) =
((\mathbf{e},\mathbf{f}), \dvalue) \in \Xset(\pi)$.  Let $j =
\pi(r)$ so that $h_{p+1}^j(\bfa)=r$.  We will show that the combined weight defined in
\eqref{e:combined.weight} satisfies
\begin{equation}\label{e:q.weights}
\q^{(\mathbf{a},\mathbf{b})} =
\begin{cases}
\q^{v_\dvalue^j(\mathbf{e})}
\q^{(\mathbf{e},\mathbf{f})}  &  \dvalue >0  \\
\q^{(\mathbf{e},\mathbf{f})}  & \dvalue = 0.
\end{cases}
\end{equation}
Once this is complete, we know from Lemma~\ref{l:aug.comaj.diff} that
$v^j(\bfa)$ is a permutation of $[r-1,r+p-1]$.  Hence,
Theorem~\ref{t:transitionBP.qanalog} follows by a straightforward
verification.

Every $((\mathbf{e},\mathbf{f}), \dvalue) \in \Xset(\pi)$ corresponds
with a pair $q\leq r$ such that $(\mathbf{e},\mathbf{f})$ is a bounded
pair for $\nu t_{q,r}$.  Recall, $\dvalue =0$ if and only if $q<r$.

In the case $\dvalue = 0$, the combined weight is preserved since the
bounded bumping algorithm preserves the ascent set of a word by
Proposition~\ref{t:little}(4).  Furthermore, the differences $a_i
-b_i$ for all $i$ are preserved by every push step in the bounded
bumping algorithm.

When $\dvalue>0$, $(\mathbf{e},\mathbf{f})$ is a bounded pair for
$\nu$.  The computation for
$\BoundedPairTransMap(\mathbf{a},\mathbf{b})$ removed a letter from
column $\dvalue$ on the last step.  The crossing removed had its right
foot on the wire labeled $r$ in the right-labeled diagram for
$\mathbf{e}$, or equivalently the wire labeled $j=\pi(r)$ when the
diagram is labeled increasing along the left side.  Therefore, the row
of the removed crossing is $h_\dvalue^j(\mathbf{e})-1$.  We also must
have $y_\dvalue^j(\mathbf{e}) = \mathbf{a}$ by definition of the
$y_\dvalue^j$ map and the fact that the bounded bumping algorithm is
reversible by Proposition~\ref{t:little}(1).  So
$h_\dvalue^j(\mathbf{e})-1 = a_\dvalue - b_\dvalue$.  In all columns
$i \neq \dvalue$, the difference $a_i -b_i$ is preserved by every push step in the
bounded bumping algorithm.  Using the
notation $$v_\dvalue^j(\mathbf{e}) = \comaj(\mathbf{a}) -
\comaj(\mathbf{e}) + h_\dvalue^j(\mathbf{e})-1,$$ we have shown
$\q^{(\mathbf{a},\mathbf{b})} = \q^{v_\dvalue^j(\mathbf{e})}
\q^{(\mathbf{e},\mathbf{f})}$.
\end{proof}

\section{Fomin-Kirillov Formulas}\label{s:fk}

Fomin and Kirillov \cite{Fomin-Kirillov} gave several
identities generalizing Macdonald's formula, and posed the
problem of finding bijective proofs.  We show that our
bijection implies a bijective proof of one of these
identities involving dominant permutations.  We first state
the identity, starting with an important special case.  In
the interest of brevity, we will assume the reader has some
familiarity with plane partitions and standard Young
tableaux.  More information on these objects may be found
in the cited references.

Let $w_{0}=[n,n-1,\ldots, 1] \in S_{n}$. The following
formula specializes to Macdonald's formula
\eqref{e:macdonald.formula} when $x=0$ and the coefficient
of the leading term is $\#R(w_{0})$.  The last quantity
equals the number of standard Young tableaux of staircase
shape with $n-1$ rows, as proved by Stanley \cite{S3}, and
later bijectively by Edelman and Greene
\cite{edelman-greene}.

\begin{thm}\cite[Thm.~1.1]{Fomin-Kirillov} \label{t:fomin.kirillov}
  We have the following identity of polynomials in $x$ for
  the permutation $w_0 \in S_n$:
\begin{equation}\label{e:fomin.kirillov.formula}
  \sum_{(a_1,\ldots, a_{\binom{n}{2}} ) \in R(w_0)} (x+a_1) \cdots (x+a_{\binom{n}{2}}) \ = \binom{n}{2} ! \prod_{1\leq i < j\leq n} \frac{2x+i+j-1}{i+j-1}.
\end{equation}
\end{thm}

The second factor on the right side of
\eqref{e:fomin.kirillov.formula} counts the number of plane
partitions with maximum entry $x$.    For a permutation
$\pi=[\pi(1),\ldots, \pi(n)]\in S_n$, write $1^{x} \times
\pi =[1,2,\ldots, x, \pi(1)+x, \pi(2)+x, \ldots,
\pi(n)+x]$.
Via theorems of Wachs \cite{wachs} and
Proctor \cite{proctor1, proctor2, koike-terada}, the second
factor on the right of \eqref{e:fomin.kirillov.formula} is
also the number of terms in the Schubert polynomial for
$1^{x} \times w$ when $x$ is a nonnegative integer.  A
bijection between the sets $R(1^{x} \times \pi)$ and $R(\pi
)$ is given by $(a_1,a_2,\ldots, a_p) \mapsto
(x+a_1,x+a_2,\ldots, x+a_p)$.

Fomin and Kirillov gave a $\q$-analog of the above identity
in which, moreover, $w_0$ is generalized to an arbitrary
dominant permutation. A \emph{dominant permutation} is one
whose code is weakly decreasing.  For any partition
$\lambda \vdash p$, let $\sigma_{\lambda}$ be the dominant
permutation in $S_{p+1}$ whose code is $\lambda$ followed
by zeros.  Let $\rpp^\lambda(x)$ be the set of \emph{weak
reverse plane partitions} whose entries are all in the
range $[0,x]$ for $x \in \mathbb{N}$.  This is the set of
$x$-bounded fillings of $\lambda$ with rows and columns
weakly increasing to the right and down.  Given a weak
reverse plane partition $P$, let $|P|$ be the sum of its
entries. Let
$$[rpp^\lambda(x)]_\q = \sum_{P\in rpp^\lambda(x)} \q^{|P|}.$$

\begin{thm}\cite[Thm.~3.1]{Fomin-Kirillov} \label{t:fomin.kirillov.2}
For any partition $\lambda \vdash p$ and its associated
dominant permutation $\sigma_\lambda$, we have the
following identity for all $x \in \mathbb{N}$:
\begin{align}\label{e:fomin.kirillov.formula.2}
  \sum_{(a_1,a_2,\ldots, a_p) \in R(\sigma_{\lambda})} \q^{\comaj(a_1,a_2,\ldots, a_p)} &[x+a_1]\cdot [x+a_2] \cdots [x+a_p] \\
  &= [p]! \   \mathfrak{S}_{1^{x} \times \sigma_{\lambda}}(1,\q,\q^2,\ldots, \q^{x+p})\\
  &= [p]! \ \q^{b(\lambda)}\   [rpp^{\lambda}(x)]_{\q}
\end{align}
where $b(\lambda) = \sum_i (i-1) \lambda_i$.
\end{thm}

The first equality is given by Macdonald's $\q$-formula.
The second follows from the theorem of Wachs~\cite{wachs}
proving that for every vexillary permutation $\pi$, its
Schubert polynomial is a flagged Schur function of shape
determined by sorting the Lehmer code of $\pi$.  (Dominant
permutations are vexillary). Using our bijection for
Macdonald's formula, we can now give a complete bijective
proof of Theorem~\ref{t:fomin.kirillov.2} as requested in
\cite[Open Problem~1]{Fomin-Kirillov}.

\begin{proof} Fix a partition $\lambda$ and $x \in \mathbb{N}$.   We
  construct a bijection $FK$ from bounded pairs for $\sigma_{\lambda}$
  to the set of
sub-staircase
 words of length $|\lambda|$ times the set of reverse plane partitions
  for $\lambda$ bounded by $x$ as follows.

\begin{enumerate}
\item Given a bounded pair $(\mathbf{a},\mathbf{b})$ for
  $\sigma_{\lambda}$, let $(\mathbf{c},D) =
  \MacdonaldMap(\mathbf{a},\mathbf{b})$ be the corresponding
  $\cd$-pair using the Macdonald Map specified in
  Section~\ref{s:bij.proof}.

\item From $D$, read the vectors of row numbers
    $\mathbf{i}_D=(i_1,\ldots,i_p)$ and diagonal numbers
    $\mathbf{r}_{D}$, as described in
    Section~\ref{s:trans}.  (Note the contrast with
    earlier proofs, where we used column numbers; the
    vector $\mathbf{i}_D$ is sometimes called a
    compatible sequence for $\mathbf{r}_{D}$ -- see
    \cite{BJS}.)

\item  Let $(P_D,Q_D)$ be the insertion tableau and
    recording tableau of the Edelman-Greene bijection
    \cite{edelman-greene} applied to the reduced word
    $\mathbf{r}_{D}$.  Let $(P_D^T,Q_D^T)$ be the
    transposes of these tableaux. (In the terminology of
    \cite{Serrano.Stump.2012}, this is Edeleman-Greene
    ``column insertion''.)

\item Let $I_D=\mathbf{i}_D\circ Q_D^T$ be the tableau
    with the same shape as $Q_D^T$ in which the entry $t$
    replaced with $i_t$, for each $t=1,\ldots,p$. By
    \cite[Theorem~3.3]{Serrano.Stump.2012}, the map $D
    \mapsto I_D$ is a weight preserving bijection from
    reduced pipe dreams for $\sigma_{\lambda}$ to column
    strict tableaux of shape $\lambda$ with row bounds
    $(1+x, 2+x, 3+x, \ldots)$.  Call this family of
    $x$-flagged tableaux $\mathcal{FT}(\lambda,x)$.
(The terminology of flagged tableaux is related to flagged Schur functions, see \cite{wachs}.)



\item From the $x$-flagged tableau $I_D$, construct the
    filling $K_D$ by subtracting $u$ from every entry in
    row $u$. Note that the rows and columns are weakly
    increasing in $K_D$ and every entry is in the
    interval $[0,x]$, so $K_D$ is a reverse plane
    partition.  Serrano and Stump prove in
    \cite{Serrano.Stump.FPSAC} that this is the bijection
    used by Fomin and Kirillov in
    Theorem~\ref{t:fomin.kirillov.2} for the second
    equality.
\end{enumerate}

The resulting map $FK: (\mathbf{a},\mathbf{b}) \to
(\mathbf{c}, K_D)$ is a bijection since each step is a
bijection.  It remains only to show that the $\q$-weight is
preserved. This follows from our bijective proof of
Theorem~\ref{t:macdonald.q.analog} and the fact that
specializing $x_i$ to $\q^{i-1}$ in the Schubert polynomial
 specializes $x^T$ to $\q^{b(\lambda)} \q^{|K_D|}$.
\end{proof}

\section{Future Directions}\label{s:future}

We briefly mention some related open problems and connections to the
literature here.
Recall that pipe dreams can be encoded as bounded pairs, but most
bounded pairs do not encode pipe dreams.  In fact, in
Section~\ref{s:trans}, we gave a simple test for this in terms of
lexicographic order on certain related pairs.  Perhaps there is
another statistic based on these pairs which could be added to
Macdonald's formula to find another generalization.

\begin{open}
  Is there an analog of the bounded pair polynomial in
  Definition~\ref{d:bounded.pair.poly} which specializes to the
  Schubert polynomial when certain parameters are set to 0?  Is there
  a common generalization for the Transition Equation for Schubert
  polynomials, bounded pairs, and its $\q$-analog?
\end{open}

Proctor's formula for plane partitions of staircase shape has a particularly
nice factored form.  This was key to the elegant formula in
\eqref{e:fomin.kirillov.formula}.  Can the staircase shape be
replaced, in any sense, with a more general partition $\lambda$?

In some sense the answer is ``no''.  There exist rather general determinantal
formulas for the partition function of the dimer model on a planar bipartite
graph~\cite{kasteleyn}, and for ensembles of nonintersecting lattice paths in a
directed acyclic graph~\cite{gessel-viennot}; it is famously possible to apply
either formula to yield a determinantal formula for reverse plane partitions of
arbitrary shapes (see, for instance, the book~\cite{bressoud} for an
introduction to this approach to plane partition enumeration).  As observed
in~\cite{Fomin-Kirillov}, typically this determinant cannot be written as a
product of nice factors.  Nonetheless, the more general Fomin-Kirillov formula
in Theorem~\ref{t:fomin.kirillov.2} makes it desirable to improve these
enumerative results as much as possible.

\begin{open}
Is there a nice formula for $|rpp^{\lambda}(x)|$ or
$[rpp^{\lambda}(x)]_\q$ for 
any large class of partitions $\lambda$,
as in the case of staircase shapes as noted in Theorem~\ref{t:fomin.kirillov}?
\end{open}



Stembridge~\cite[Thm. 1.1]{stembridge} gives a formula for a weighted
enumeration of maximal saturated chains in the Bruhat order for any
Weyl group which is very similar to Macdonald's formula.  This formula
is related to the study of degrees of Schubert varieties,
see~\cite{chevalley, postnikov-stanley}, and has no obvious direct
connection to Theorem~\ref{t:macdonald}.  Stanley \cite[Equation
(23)] {Stanley.perms} stated the following version of Stembridge's
weighted enumeration formula in the case of $S_{n}$ and noted the
similarity to Macdonald's formula.  Given $w \in S_n$ of length $p$,
let
\[
\begin{split}
\mathcal{T}(w) := \Bigl\{&\bigl((i_1, j_1), (i_2, j_2), \ldots, (i_p, j_p)\bigr) \;:\; w =
t_{i_1,j_1}t_{i_2,j_2}\cdots t_{i_p,j_p} \\ &\quad\text{and }
\ell(t_{i_1,j_1}t_{i_2,j_2}\cdots t_{i_k,j_k}) = k \text{ for all } 1 \leq k
\leq p \Bigr\}.
\end{split}
\]
\begin{thm}
For $w=w_0 \in S_n$,  we have
\begin{equation}
\label{eqn:mysterious}
\sum_{((i_1, j_1), (i_2, j_2), \ldots, (i_p, j_{p})) \in \mathcal{T}(w_0)}
(j_1-i_1)(j_2-i_2) \cdots(j_{p} - i_{p}) = \binom{n}{2}!.
\end{equation}
\end{thm}

\begin{open} 
Can~\eqref{eqn:mysterious} be proven bijectively, using a similar
technique to our $\MacdonaldMap$  bijection?
\end{open}

The left  side of~\eqref{eqn:mysterious} has a natural
interpretation in terms of pairs $(\bfa, \mathbf{b})$ where $\bfa$ is
a word of transpositions $(i_k,j_k)$, and $\mathbf{b} = (b_1, \ldots,
b_p)$ is a sort of ``bounded word'' with bounds $i_k \leq b_k <
j_k$.  However, no analogue of Little's bumping map is known for
maximal saturated Bruhat chains.  Worse,~\eqref{eqn:mysterious} is
only known to hold for the longest word $w_0$.  It would be necessary
to find a generalization of~\eqref{eqn:mysterious} to other $w$ before
the strategy outlined in this paper could apply.

The Grothendieck polynomials are the $K$-theory analog of Schubert polynomials
for the flag manifolds \cite{lascoux.schutzenberger.grothendieck}.  There is a
Transition Formula for these polynomials
\cite{lascoux.grothendieck.transition}.  Anders Buch asked the following
question.

\begin{open}
What is the analog of Macdonald's formula for Grothendieck polynomials
and what is the corresponding bijection?
\end{open}
Fomin-Kirillov~\cite{fomin-kirillov-yang-baxter} state a Macdonald-type formula
for the longest word $w_0$, as a corollary of their work on the degenerate Hecke
algebra.  Curiously the Stirling numbers of the second kind appear on the right
side of the formula.  There are also some partial recent partial results on
this open problem due to Reiner, Tenner and Yong \cite{RTY.2016}.  In
particular, see their Definition 6.2 and Conjecture 6.3.

\section{Appendix}\label{s:appendix}

For $\mathbf{a} = (4)$ with $\comaj(\mathbf{a}) = 0$ and $j=5$,  we have $v^5(\mathbf{a})=(4, 3)$.
$$\begin{array}{|c|c|c|c|c|c|c|}
\hline
i & h^5_i(\mathbf{a}) &  \text{insert}  & y_i^5(\mathbf{a}) & \comaj(y_i^5(\mathbf{a})) &h^5_i(\mathbf{a})-1  & v^5_i(\mathbf{a}) \\
\hline
0 & 5 & (4 4) & (5 4) & 0  & 4 & 4
\\
1 & 4 & (4 3) & (5 4) & 0  & 3 & 3
\\
\hline
\end{array}
$$

For $\mathbf{a} = (4, 3)$ with $\comaj(\mathbf{a}) = 0$ and $j=5$,  we have $v^5(\mathbf{a})=(4, 3, 2)$.
$$\begin{array}{|c|c|c|c|c|c|c|}
\hline
i & h^5_i(\mathbf{a}) &  \text{insert}  & y_i^5(\mathbf{a}) & \comaj(y_i^5(\mathbf{a}))  &h^5_i(\mathbf{a})-1  & v^5_i(\mathbf{a}) \\
\hline
0 & 5 & (4 4 3) & (5 4 3) & 0 & 4 & 4
\\
1 & 4 & (4 3 3) & (5 4 3) & 0 & 3 & 3
\\
2 & 3 & (4 3 2) & (5 4 3) & 0 & 2 & 2
\\
\hline
\end{array}
$$

For $\mathbf{a} = (4, 3, 5)$ with $\comaj(\mathbf{a}) = 2$ and $j=5$,  we have $v^5(\mathbf{a})=(5, 4, 3, 2)$.
$$\begin{array}{|c|c|c|c|c|c|c|}
\hline
i & h^5_i(\mathbf{a}) &  \text{insert}  & y_i^5(\mathbf{a}) & \comaj(y_i^5(\mathbf{a})) &h^5_i(\mathbf{a})-1  & v^5_i(\mathbf{a}) \\
\hline
0 & 5 & (4 4 3 5) & (5 4 3 5) & 3  & 4 & 5
\\
1 & 4 & (4 3 3 5) & (5 4 3 5) & 3  & 3 & 4
\\
2 & 3 & (4 3 2 5) & (5 4 3 5) & 3  & 2 & 3
\\
3 & 3 & (4 3 5 2) & (5 4 5 3) & 2  & 2 & 2
\\
\hline
\end{array}
$$

For $\mathbf{a} = (4, 3, 5, 6)$ with $\comaj(\mathbf{a}) = 5$ and $j=5$,  we have $v^5(\mathbf{a})=(6, 5, 4, 3, 2)$.
$$\begin{array}{|c|c|c|c|c|c|c|}
\hline
i & h^5_i(\mathbf{a}) &  \text{insert}  & y_i^5(\mathbf{a}) & \comaj(y_i^5(\mathbf{a})) &h^5_i(\mathbf{a})-1  & v^5_i(\mathbf{a}) \\
\hline
0 & 5 & (4 4 3 5 6) & (5 4 3 5 6) & 7  & 4 & 6
\\
1 & 4 & (4 3 3 5 6) & (5 4 3 5 6) & 7  & 3 & 5
\\
2 & 3 & (4 3 2 5 6) & (5 4 3 5 6) & 7  & 2 & 4
\\
3 & 3 & (4 3 5 2 6) & (5 4 5 3 6) & 6  & 2 & 3
\\
4 & 3 & (4 3 5 6 2) & (5 4 5 6 3) & 5  & 2 & 2
\\
\hline
\end{array}
$$

For $\mathbf{a} = (4, 3, 5, 6, 4)$ with $\comaj(\mathbf{a}) = 5$ and $j=5$,  we have $v^5(\mathbf{a})=(6, 5, 4, 3, 7, 2)$.
$$\begin{array}{|c|c|c|c|c|c|c|}
\hline
i & h^5_i(\mathbf{a}) &  \text{insert}  & y_i^5(\mathbf{a}) & \comaj(y_i^5(\mathbf{a}))  &h^5_i(\mathbf{a})-1  & v^5_i(\mathbf{a}) \\
\hline
0 & 5 & (4 4 3 5 6 4) & (5 4 3 5 6 4) & 7 & 4 & 6
\\
1 & 4 & (4 3 3 5 6 4) & (5 4 3 5 6 4) & 7 & 3 & 5
\\
2 & 3 & (4 3 2 5 6 4) & (5 4 3 5 6 4) & 7 & 2 & 4
\\
3 & 3 & (4 3 5 2 6 4) & (5 4 5 3 6 4) & 6 & 2 & 3
\\
4 & 3 & (4 3 5 6 2 4) & (5 4 5 6 3 4) & 10 & 2 & 7
\\
5 & 3 & (4 3 5 6 4 2) & (4 3 5 6 4 3) & 5 & 2 & 2
\\
\hline
\end{array}
$$

For $\mathbf{a} = (4, 3, 5, 6, 4, 3)$ with $\comaj(\mathbf{a}) = 5$ and $j=5$,  we have $v^5(\mathbf{a})=(6, 5, 4, 3, 7, 8, 9)$.
$$\begin{array}{|c|c|c|c|c|c|c|}
\hline
i & h^5_i(\mathbf{a}) &  \text{insert}  & y_i^5(\mathbf{a}) & \comaj(y_i^5(\mathbf{a})) &h^5_i(\mathbf{a})-1  & v^5_i(\mathbf{a}) \\
\hline
0 & 5 & (4 4 3 5 6 4 3) & (5 4 3 5 6 5 4) & 7 & 4 & 6
\\
1 & 4 & (4 3 3 5 6 4 3) & (5 4 3 5 6 5 4) & 7 & 3 & 5
\\
2 & 3 & (4 3 2 5 6 4 3) & (5 4 3 5 6 5 4) & 7 & 2 & 4
\\
3 & 3 & (4 3 5 2 6 4 3) & (5 4 5 3 6 5 4) & 6 & 2 & 3
\\
4 & 3 & (4 3 5 6 2 4 3) & (5 4 5 6 3 5 4) & 10 & 2 & 7
\\
5 & 3 & (4 3 5 6 4 2 3) & (5 4 5 6 5 3 4) & 11 & 2 & 8
\\
6 & 4 & (4 3 5 6 4 3 3) & (5 4 5 6 5 3 4) & 11 & 3 & 9
\\
\hline
\end{array}
$$

For $\mathbf{a} = (4, 3, 5, 6, 4, 3, 5)$ with $\comaj(\mathbf{a}) = 11$ and $j=5$,  we have $v^5(\mathbf{a})=(7, 6, 5, 4, 8, 9, 10, 3)$.
$$\begin{array}{|c|c|c|c|c|c|c|}
\hline
i & h^5_i(\mathbf{a}) &  \text{insert}  & y_i^5(\mathbf{a}) & \comaj(y_i^5(\mathbf{a})) &h^5_i(\mathbf{a})-1  & v^5_i(\mathbf{a}) \\
\hline
0 & 5 & (4 4 3 5 6 4 3 5) & (5 4 3 5 6 5 4 5) & 14 & 4 & 7
\\
1 & 4 & (4 3 3 5 6 4 3 5) & (5 4 3 5 6 5 4 5) & 14 & 3 & 6
\\
2 & 3 & (4 3 2 5 6 4 3 5) & (5 4 3 5 6 5 4 5) & 14 & 2 & 5
\\
3 & 3 & (4 3 5 2 6 4 3 5) & (5 4 5 3 6 5 4 5) & 13 & 2 & 4
\\
4 & 3 & (4 3 5 6 2 4 3 5) & (5 4 5 6 3 5 4 5) & 17 & 2 & 8
\\
5 & 3 & (4 3 5 6 4 2 3 5) & (5 4 5 6 5 3 4 5) & 18 & 2 & 9
\\
6 & 4 & (4 3 5 6 4 3 3 5) & (5 4 5 6 5 3 4 5) & 18 & 3 & 10
\\
7 & 4 & (4 3 5 6 4 3 5 3) & (4 3 5 6 4 3 5 4) & 11 & 3 & 3
\\
\hline
\end{array}
$$

\section*{Acknowledgments}\label{s:ack}
Many thanks to Connor Ahlbach, Sami Assaf, Anders Buch, Sergey Fomin,
Ira Gessel, Zachary Hamaker, Avi Levy, Peter McNamara, Maria Monks Gillespie, Alejandro Morales, Richard Stanley, Dennis
Stanton, Joshua Swanson,  and Marisa Viola for helpful discussions on this work.

\bibliographystyle{siam} \bibliography{mrwf}

\end{document}